\documentclass{alggeom}
\usepackage{amsmath}
\usepackage{amscd}
\usepackage{amssymb}
\usepackage{color}
\usepackage[matrix, arrow]{xy}
\usepackage{verbatim}
\usepackage{hyperref}

\newtheorem{lemma}{Lemma}[section]
\newtheorem{prop}[lemma]{Proposition}
\newtheorem{thm}[lemma]{Theorem}
\newtheorem{cor}[lemma]{Corollary}

\newtheorem*{theorem*}{Theorem}
\newtheorem{introthm}{Theorem}

\theoremstyle{definition}
\newtheorem{defn}[lemma]{Definition}

\theoremstyle{remark}
\newtheorem{rem}[lemma]{Remark}
\newtheorem{ex}[lemma]{Example}
\newtheorem{constr}[lemma]{Construction}

\newcommand{\Spec}{\mathrm{Spec}}
\newcommand{\Ext}{\mathrm{Ext}}
\newcommand{\Mor}{\mathrm{Mor}}
\newcommand{\Ker}{\mathrm{Ker}}

\newcommand{\isom}{\cong}
\newcommand{\id}{\mathrm{id}}

\newcommand{\ohne}{\smallsetminus}
\newcommand{\tensor}{\otimes}

\newcommand{\Ah}{\mathcal{A}}

\newcommand{\Oh}{\mathcal{O}}
\newcommand{\Fh}{\mathcal{F}}
\newcommand{\Ih}{\mathcal{I}}

\newcommand{\Q}{\mathbb{Q}}
\newcommand{\Z}{\mathbb{Z}}

\newcommand{\C}{\mathbb{C}}

\newcommand{\A}{\mathbb{A}}

\newcommand{\uOmega}{\underline{\Omega}}

\newcommand{\Sch}{\mathrm{Sch}}
\newcommand{\Var}{\mathrm{Var}}
\newcommand{\Sm}{\mathrm{Sm}}
\newcommand{\Zar}{\mathrm{Zar}}
\newcommand{\et}{\mathrm{et}}
\newcommand{\eh}{\mathrm{eh}}
\newcommand{\cdh}{\mathrm{cdh}}
\newcommand{\h}{\mathrm{h}}
\newcommand{\qfh}{\mathrm{qfh}}
\newcommand{\red}{\mathrm{red}}
\newcommand{\reg}{\mathrm{reg}}
\newcommand{\sn}{\mathrm{sn}}
\newcommand{\dR}{\mathrm{dR}}
\newcommand{\sing}{\mathrm{sing}}
\newcommand{\an}{\mathrm{an}}
\newcommand{\Sh}{\mathrm{Sh}}

\newcommand{\torsion}{\text{torsion}}
\newcommand{\basehtop}[1]{{\h/#1}}
\newcommand{\Xhtop}{\basehtop{X}}
\newcommand{\Yhtop}{\basehtop{Y}}
\newcommand{\Zhtop}{\basehtop{Z}}

\begin{document}

\title{Differential forms  in the $\h$-topology}
\author{Annette Huber}
\email{annette.huber@math.uni-freiburg.de}
\address{Math. Institut, Universit\"at Freiburg, Eckerstr. 1, 79104 Freiburg, Germany.
}
\author{Clemens J\"order}
\email{c.joerder@web.de}
\address{Math. Institut, Universit\"at Freiburg, Eckerstr. 1, 79104 Freiburg, Germany.
}
\classification{14F10 (primary), 14F05, 14J17, 14F40(secondary)}
\keywords{differential forms, $h$-topology, $klt$-spaces, Du Bois complex, de Rham cohomology}
\thanks{The second author gratefully acknowledges support by the DFG-Forschergruppe 790
  ``Classification of Algebraic Surfaces and Compact Complex Manifolds'' and the Graduiertenkolleg 1821 ``Cohomological Methods in Geometry''.}

\begin{abstract}
We study sheaves of differential forms and their cohomology in the $\h$-topology.
This allows one to extend standard results from the case of smooth varieties to the general case. As a first application we explain the case of singularities arising in the Minimal Model Program. As a second application we consider de Rham cohomology.
\end{abstract}
\maketitle

\setcounter{tocdepth}{1}

\section*{Introduction}
The aim of this note is to propose a new extension of the theory of differential forms to the case of singular varieties in characteristic zero and to illustrate that it has very good properties; unifying  a number of ad hoc approaches and allowing a more conceptual understanding of results in the literature.

Differential forms play a key role in the study of local and global properties of manifolds and non-singular algebraic varieties. This principle is confirmed for example by the period isomorphism between algebraic de Rham cohomology and singular cohomology, or the classification of singularities arising from the Minimal Model Program in terms of extension properties of differential forms of top degree.

It is well-known that the theory of K\"ahler differentials is not well-behaved in the singular case. 
Various competing generalizations were introduced. In any case the definition can be traced back to the non-singular case:
\begin{itemize}
  \item \emph{K\"ahler differential forms} $\Omega^p_X$ on $X$ and their torsion-free counterpart $\Omega^p_X/\text{tor}$ are obtained as quotients of $\Omega^p_Y$ for an ambient non-singular space $Y\supset X$.
  \item \emph{Differential forms of first kind} on an irreducible variety $X$ are differential forms on a log resolution $Y\to X$ (see \cite{SvS85} (1.2)).
  \item \emph{Reflexive differential forms} $\Omega^{[p]}_X$  on a normal variety $X$ are differential forms on the regular locus $X^\reg$ (see \cite{Kni73}, \cite{LW09}, \cite{GKKP11})
\item Using simplicial hyperresolutions Du Bois \cite{DuB81} defines complexes of coherent sheaves and in this way ''localizes Hodge theory''.
\end{itemize}
The purpose of this paper is to introduce a new competitor to this field: \emph{$\h$-dif\-fe\-ren\-tial  forms} $\Omega^p_\h$. We give three characterizations of very different flavor:
\begin{enumerate}
\item
They are the outcome of the sheafification of K\"ahler differential forms with respect to the $\h$-topology on the category of schemes introduced by Voevodsky in \cite{Voe96} (see Definition~\ref{defn_h}).
\item  They have a simple characterization in birational geometry: Given a variety $X$ we choose arbitrary resolutions $X'\to X$ and $\phi:X''\to X'\times_X X'$. Then pulling back yields a bijection between the set of $\h$-differential forms on $X$ and the set of K\"ahler differential forms on the resolution $X'$ such that the two pullbacks to $X''$ coincide. In other words (see Remark \ref{rem_simple-formula}),
\[\Omega^p_\h(X)\,\cong\, \{\alpha\in\Omega^p_X(X')|\, (\text{pr}_1\circ\phi)^*\alpha = (\text{pr}_2\circ\phi)^*\alpha \} .\]
\item To give an $\h$-differential form on a variety $X$ is equivalent to give, in a compatible way, for any morphism $Y\to X$ from a non-singular variety to $X$ a K\"ahler differential form on $Y$. More precisely:
\end{enumerate}

\begin{introthm}[(Section \ref{sec_h})]\label{thm1}
Let $k$ be a field of characteristic zero and $X$ a separated scheme of finite type over $k$. Then
\[ \Omega^p_\h(X) \,\cong\, \left\{ (\alpha_f)_{f:Y\to X}\in\prod_{f:Y\to X\atop \text{$Y$ smooth}}\Omega^p_Y(Y)  \quad    \Big|\quad \begin{array}{l}\xymatrix{Y' \ar[rd]^{f'} \ar[d]_\phi & \\ Y \ar[r]_f & X }\end{array}\Rightarrow \phi^*\alpha_{f}=\alpha_{f'}\right\}. \]
\end{introthm}

Let us give three reasons why one should consider the $\h$-topology. First, any scheme is $\h$-locally smooth by Hironaka's theorem, an obvious technical advantage.

Second, K\"ahler differential forms on non-singular schemes turn out to satisfy $\h$-descent (see Theorem \ref{thm_hdescent}). A variant of this result  was first shown by Lee  in \cite{Lee09}. An analogous statement has already been observed in the case of the coarser $\eh$-topology by Geisser in \cite{Gei06} and in the even coarser $\cdh$-topology by Corti\~nas, Haesemeyer, Walker and Weibel in \cite{CHWaWei11}.

Third, in contrast to the $\cdh$ or $\eh$-topology, all proper surjective morphisms and all flat covers are $\h$-covers. Recall that proper covers in the context of de Rham cohomology were introduced long ago by Deligne in \cite{HodgeIII} in order to extend the period isomorphism to the singular case. We would like to emphasize the flexibility gained by using arbitrary $\h$-covers. In many cases, technical difficulties disappear thanks to the machinery.

These technical advantages allow us to prove invariance of $\h$-differentials for maps with rationally chain connected fibers (Theorem \ref{thm_rcc}). Together with the extension theorem of Greb, Kebekus, Kov\'acs and Peternell in \cite{GKKP11}, this implies the following result for varieties whose singularities arise in the Minimal Model Program.

\begin{introthm}[(Section \ref{sec:5})]\label{thm2}
On a klt base space: $\Omega^p_\h(X)=\Omega^{[p]}_X(X)$.
\end{introthm}

The special case of normal toric varieties has already been proved in \cite{CHSWaWei00} Theorem 4.1, see Remark~\ref{rem_toric} for more details. By Theorem~2 we obtain a more conceptual explanation  for the existence of pull-back maps for reflexive differential forms on klt base spaces (see Corollary~\ref{cor-keb-pullback}). This is the main result in \cite{Keb12}. On the other hand, \cite{Keb12} and Theorem 1 together imply Theorem 2.

Let us now turn to the study of cohomology of $\h$-differentials. The natural notion is cohomology of $\h$-sheaves. This does not change anything in the smooth case:

\begin{introthm}[(Section \ref{sec_coho})]\label{thm3}
On a smooth variety: $H^i_\h(X,\Omega^p_\h)=H^i_\Zar(X, \Omega^p_X)$.
\end{introthm}

The main computational tool, the blow-up sequence, allows one easily to compute cohomology of singular varieties from the smooth case. 

Subsequently we will analyze in more detail the relation between the $\h$-topology on the category of schemes over a scheme $X$ and the Zariski topology on $X$. This is useful in two ways: 

First, as already realized by Lee (\cite{Lee09}), the new point of view offers a new perspective on the Du Bois complex of a variety. It turns out to be the derived push-forward of the $\h$-differential forms considered as a complex of sheaves in the $\h$-topology (Theorem \ref{thm_h-is-dub}). 
To illustrate the advantage of our language we deduce subsequently a number of well-known properties of the Du Bois complex. 
In contrast to \cite{DuB81} and \cite{GNPP88},
we avoid any use of simplicial or cubic hyperresolutions in the construction.

Analogous results in terms of the $\cdh$-topology were also shown by Corti\~nas, Haesemeyer, Schlichting, Walker and Weibel as a byproduct of their work on homotopy invariance of algebraic $K$-theory, see their series of papers \cite{CHSW08}, \cite{CHWei08}, \cite{CHSWaWei00}, \cite{CHWaWei10}, \cite{CHWaWei11}, \cite{CHWaWei13}. 

This leads to the last application. Hypercohomology of the complex of $\h$-sheaves $\Omega^p_\h$ gives a simple definition of {\em algebraic de Rham cohomology}. It agrees with the other definitions in the literature. 
The question to what extend the analytified sheaves of $\h$-differential forms resolve the sheaf of locally constant functions is treated in the dissertation \cite{Jor14} of the second author. 

As an application of our techniqeus, we use the machinery of $\h$-differential forms to construct the relative de Rham cohomology associated with a closed subscheme $Z\subset X$. 
It is the counterpart of relative singular 
cohomology and needed in the study of the period isomorphism and the period numbers of a general variety, e.g. in
the work of Kontsevich and Zagier on periods \cite{KoZ02}, see also \cite{HubMSt11}.  Its existence and properties have been known to experts, but
we are not aware of a good reference. 
The obvious definition via the shift of a reduced cone of hyperresolutions
makes establishing the properties very difficult. It is not even clear to us how
to write down the K\"unneth morphism.
One possible approach is by using Saito's theory of Hodge modules, another by a systematic use of the de Rham realization
of triangulated motives (see \cite{Hreal} and \cite{Hreal2}, or \cite{LW09}).
Using $\h$-differentials we are able to write down a simple definition
and give straightforward algebraic proofs for basic features such as long exact sequences associated with triples, excision and the K\"unneth formula (Subsection \ref{ssec_rel-de-rham}). 

Beilinson in his approach to $p$-adic Hodge theory (see \cite{B12}) also 
uses the
$h$-topology in order to study the de Rham complex. His construction is a lot 
more subtle. He is working over the ring of integers of a local field of mixed characteristic, i.e., with integral (or rather $p$-adic) coefficients. Moreover, 
his $\mathcal{A}_\dR$ is a projective system of complexes of $\h$-sheaves which is built on Illusie's cotangent complex instead of the cotangent space. We have not tried to work out the explicit relation to $\Omega^*_\h$ on the generic fibre. 

Using the $\h$-topology in the context of de Rham cohomology is quite natural: the de Rham complex is a homotopy invariant complex of sheaves of 
$\Q$-vector spaces with transfers.
From the general motivic machinery we learn that its cohomology can be equivalently treated in the $\h$-topology without transfers. One point we want to demonstrate in our paper is that the $\h$-sheafification of the individual $\Omega^p$ is also very useful. 

The present paper concentrates on differential forms rather than developing a 
full-blown six functor formalism for $\Oh$-modules in the $\h$-topology. 
We have refrained from dealing with differential forms with log poles or with twists by line bundles. What is also missing is
a discussion of Grothendieck duality where differential forms also play a key role. We work in characteristic zero throughout and hope that a modified definition would also work in positive characteristic.
We hope that these aspects will be developed in the future.
\subsection*{Outline of the paper}
The paper consists of two parts. Section~\ref{sec:2} through \ref{sec:5} discuss sheaves of differential forms in various topologies. Sections~\ref{sec_coho} and~\ref{sec:7} deal with cohomology groups and objects in the derived category.

The goal of sections \ref{sec:2} and \ref{sec_h} is to establish $\h$-descent for differential forms. Section~\ref{sec:2} recalls the results of Geisser \cite{Gei06} in his $\eh$-topology 
and establishes additional properties of differentials forms in the $\eh$-topology. Section \ref{sec_h} reviews Voevodsky's $\h$-topology and deduces Theorem \ref{thm1} mentioned above.

In Section \ref{sec:4}, we make the definition of $\Omega^p_\h(X)$ explicit in a number of
cases: for $p=0$, $p=\dim X$ or when $X$ has special types of singularities.

Section \ref{sec:5} treats the case of klt-singularities. Theorem \ref{thm2} is deduced from the key result on the invariance of
$\h$-differentials for maps with rationally chain connected fibers.

We then turn in Section \ref{sec_coho} to cohomology of sheaves in the $\h$-topology in general and to cohomology of $\h$-differential forms. 

Finally, in Section \ref{sec:7}, we consider the de Rham complex of $\h$-differential forms. We recover the Du Bois complex in terms of $\h$-differential forms. The section concludes with a simple description of relative algebraic de Rham cohomology in terms of $\h$-differential forms.

\begin{acknowledgements}
We are particularly thankful to Stefan Kebekus who suggested that we should work together on this project. He was the one who asked about the relation to the Du Bois complex. We are very grateful to Matthias Wendt for his careful reading of a first version of our article. We also thank them and our other colleagues at Freiburg
University for comments and explanations: Daniel~Greb, Patrick Graf, and Wolfgang Soergel.
\end{acknowledgements}

\section{Setting and Notation}\label{sec:1}
We fix a field $k$ of characteristic zero. By scheme we mean a separated scheme of finite type over $k$. By variety we mean a reduced separated scheme of finite type over $k$. We denote by $\Sch$, $\Var$ and $\Sm$ the categories of $k$-schemes, $k$-varieties and smooth $k$-varieties, respectively. 

A {\em resolution} of an irreducible variety $X$ is a proper birational morphism $X'\to X$ from a smooth variety $X'$ to $X$. A {\em resolution} of a variety $X$ is a morphism $X'\to X$ where $X'$ is the disjoint union of resolutions of the irreducible components of $X$. 

It $t$ is a Grothendieck topology, we denote by $\Sch_t$, $\Var_t$, $\Sm_t$ the site defined by $t$ and by $(\Sch_t)^\sim$, $(\Var_t)^\sim$ and $(\Sm_t)^\sim$ the topos of sheaves of sets on $\Sch_t$, $\Var_t$ and $\Sm_t$, respectively.
We are going to consider the cases $\Zar$ (Zariski topology), $\et$ (\'etale topology), $\eh$ (\'etale h-topology, see Definition \ref{defn-eh}) and $\h$ (h-topology, see Definition \ref{defn_h}).

If $\Fh$ is a $t$-sheaf of abelian groups in some Grothendieck topology on $\Sch$ and $X\in\Sch$, then we write
\[ H^i_t(X,\Fh)\]
for the $i$-th derived functor of $\Gamma(X,\cdot):$ evaluated on $\Fh$.
\begin{defn}\label{defn_Z(X)}Let $X\in\Sch$, and $t$ some Grothendieck topology on $\Sch$. We write $\Z_t(X)$ for the $t$-sheafification of the presheaf
\[ T\mapsto \Z[X(T)]\]
where $\Z[S]$ denotes the free abelian group generated by $S$.
\end{defn}
Recall that
\[ H^i_t(X,\Fh)=\Ext^i(\Z_t(X),\Fh)\]
for all t-sheaves $\Fh$.

For a $k$-scheme $X$ let $\Omega^1_X$ be the Zariski-sheaf of $k$-linear K\"ahler differentials on $X$. For $p\geq 0$, let
$\Omega^p_X$ be the $p$-th exterior power of $\Omega^1_X$ in the category of $\Oh_X$-modules. 
We denote $\Omega^p$ the sheaf 
\[ X\mapsto \Omega^p_X(X)\]
on the big Zariski-site on $\Sch_k$. 
The usual differential 
\[ d:\Omega^p\to\Omega^{p+1}\]
turns it into a differential graded algebra $\Omega^\bullet$.
If $t$ is another topology on $\Sch$, we denote by $\Omega^\bullet_t$ the
sheafification in the $t$-topology.

We are also going to consider Zariski-differentials, studied e.g. in
\cite{Kni73} and \cite{GKKP11}. We follow the notation of the second reference.

\begin{defn}\label{defn_reflexive}Let $X$ be a normal variety, $j:X^\reg\to X$ be the inclusion of
the regular locus. We call
\[ \Omega^{[p]}_X=j_*\Omega^p_{X^\reg}\]
the {\em sheaf of Zariski differentials on $X$} or {\em sheaf of reflexive differentials}.
\end{defn}


\section{Differential forms in the $\eh$-topology}\label{sec:2}
We review the $\eh$-topology introduced by Geisser in \cite{Gei06}. It is
a twin of the cdh-topology introduced by Voevodsky in \cite{Voe96}. The relation of the $\eh$-topology to the \'etale topology is the same as the relation of the 
Nisnevich topology to the cdh-topology. For our purposes it makes no difference which to use.
We consider differential forms in the $\eh$-topology.

\begin{defn}[(\cite{Gei06} Definition 2.1)]\label{defn-eh}The {\em $\eh$-topology} on the category $\Sch$ of separated schemes of finite type over $k$ is the Grothendieck topology generated by the following coverings:
\begin{itemize}
\item \'etale coverings
\item abstract blow-ups: assume that we have a Cartesian square
\[\begin{CD}
Z'@>>> X'\\
@Vf'VV@VVfV\\
Z@>>> X
\end{CD}\]
where $f$ is proper, $Z\subset X$ a closed subscheme  and $f$ induces an isomorphism
$X'\ohne Z'\to X\ohne Z$; then $(X'\to X, Z\to X)$ is a covering.
\end{itemize}
\end{defn}

\begin{ex}\label{lemma_propereh}
\begin{enumerate}
\item Let $f:X_{\red}\subset X$ be the reduction. Then $f$ is an $\eh$-cover
(with $X'=\emptyset$, $Z=X_\red$). Hence every scheme is $\eh$-locally reduced. 
\item \label{item_2.3.1}
Every proper morphism $X'\to X$ such that for every $x\in X$ there is a point in $p^{-1}(x)$ with the same residue field as $x$ is an eh-covering by \cite{Gei06} Lemma 2.2. A special case is a blow-up of a smooth variety in a smooth center.

\end{enumerate}
\end{ex}

\begin{defn} A covering as in Example \ref{lemma_propereh}  (\ref{item_2.3.1}) is called a {\em proper eh-covering}. 
\end{defn}

\begin{constr}\label{constr_eh-locally-smooth}
By resolution of singularities, every $X\in \Sch$ has a proper $\eh$-cover $\{\tilde{X}_i\to X\}_{i\in I}$ by smooth varieties $\tilde{X}_i$. We spell out the algorithm. 
\begin{enumerate}
\item $X\in\Sch$ has a proper $\eh$-cover $X$ by  $X_\red$.
\item Let $X=\bigcup_{i=1}^n X_i\in\Var$ with $X_i$ irreducible. Then
\[ \left\{ \coprod_{i=1}^nX_i\to X, \bigcup_{i<j}X_i\cap X_j\right\}\]
is an abstract blow-up and hence a proper $\eh$-cover.
\item Let $X$ be an irreducible variety. By resolution of singularities there is
a birational proper map $\pi:\tilde{X}\to X$ with $\tilde{X}$ smooth. Let
$Z\subset X$ be the image in $X$ of the exceptional locus of $\pi$. Then
$\{\tilde{X},Z\}$ is an abstract blow-up and hence a proper $\eh$-cover.
\item Let $Z$ be as in the last step.
By induction on the dimension, there is a proper $\eh$-cover of $Z$
by smooth varieties.
\end{enumerate}
\end{constr}

\begin{defn}For $p\geq 0$, let $\Omega_\eh^p$ be the $\eh$-sheafification of
the presheaf
\[ X\mapsto \Omega^p(X)\]
on the category $\Sch$. We call the elements of $\Omega_\eh^p(X)$ {\em eh-differentials on $X$}.
\end{defn}



\begin{prop}[(\cite{Gei06} Thm. 4.7)]\label{prop_smooth} Let $X$ be a smooth variety. Then
\[ \Omega^p(X)=\Omega^p_\eh(X)\ .\]
\end{prop}

\begin{rem}
The proof of Proposition~\ref{prop_smooth} in~\cite{Gei06} Thm. 4.7 appears rather technical. 
The main mathematical content of Proposition~\ref{prop_smooth} is that $\Omega^p$ satisfies the sheaf property for $\eh$-covers of a smooth variety by a smooth variety. This fact might already be known to the reader in special cases such as an \'etale cover of a smooth variety (see \cite{Milne} III Prop. 3.7), or a blowup $X_Z\to X$ of a smooth variety along a smooth subvariety (observe that $\Omega^p(X_Z)\cong\Omega^p(X)$ by \cite{Gros} Chapitre IV Th\'eor\`eme 1.2.1 in this case). The general case follows from these  by resolution of singularities.
\end{rem}

In order to connect our results with an alternative version considered elsewhere in the literature, we record another comparison result.

\begin{cor}\label{cor_ehcdh}
Let $\Omega^p_\cdh$ be the sheafification of $\Omega^p$ with respect to
the $\cdh$-topology (generated by abstract blow-ups and Nisnevich covers). Then
for any $X\in\Sch$,
\[ \Omega^p_\cdh(X)=\Omega^p_\eh(X)\ .\]
\end{cor}
\begin{proof}
Note that the $\cdh$-topology is coarser than the $\eh$-topology but still contains all abstract blow-ups. Thus Construction \ref{constr_eh-locally-smooth} shows that all schemes are $\cdh$-locally smooth. Hence it suffices to consider the case $X$ smooth.

We combine  the comparison theorem for the $\eh$-topology with its analogue in the $\cdh$-topology in \cite{CHWaWei11} Lemma 2.9. 
\end{proof}

We now establish properties of $\eh$-differentials needed later on.

\begin{lemma}\label{cor_injdominant} 
Let $f:X\to Y$ be a dominant morphism. Then 
\[ \Omega^p_\eh(Y)\hookrightarrow\Omega^p_\eh(X)\]
is injective.
\end{lemma}
\begin{proof}
First assume that $Y$ and $X$ are smooth and irreducible. By Proposition~\ref{prop_smooth}
we have to consider $\Omega^p$. As $\Omega^p$ is a vector bundle on
$X$, we have $\Omega^p(X)\subset\Omega^p(k(X)/k)$.
As $f$ is a dominant morphism between irreducible schemes of characteristic zero, we have $\Omega^p(k(Y)/k)\subset\Omega^p(k(X)/k)$.  This settles the smooth case.

In general, we may assume that $Y=Y_\red$. 
Let $\alpha\in\Omega^p_\eh(Y)$ and write $\alpha|_T$ for the pullback of $\alpha$ by a map $T\to Y$. We assume that $\alpha|_X=0$ and seek to prove that $\alpha=0$. By Construction~\ref{constr_eh-locally-smooth} it suffices to show that $\alpha|_T=0$ for all $T\to Y$ where $T$ is smooth and irreducible.

Choose a resolution $S'\to (X\times_YT)_\red$ of singularities and let $S\subset S'$ be an irreducible component dominating $T$. Then $\alpha|_{X}=0$ implies that $\alpha|_S=0$. This implies $\alpha|_T=0$ by the smooth case.
\end{proof}




\begin{lemma}\label{lemma_Gdescent} Let $X$ be normal and irreducible, $K/K(X)$ a Galois extension with Galois group 
$G$ and $Y$ the normalization of $X$ in $K$. Then
\[ \Omega^p_\eh(X)=\Omega^p_\eh(Y)^G\]
and $\Omega^p_\eh$ has descent for $\pi:Y\to X$, i.e, the sequence
\[ 0\to \Omega^p(X)\to \Omega^p_\eh(Y)\xrightarrow{\text{pr}_1^* - \text{pr}_2^*}\Omega_\eh(Y\times_XY)\]
is exact.
\end{lemma}
\begin{proof}

If $X$ is smooth and $Y\to X$ is \'etale, then $Y$ is also smooth
and the lemma holds because $\Omega^p$ has \'etale descent. In the general case, let $U\subset X$ be a smooth open  subscheme over which $\pi$ is \'etale. We get a commutative
diagram
\[\begin{CD}
0@>>>\Omega^p(\pi^{-1}U)^G@>>>\Omega^p(\pi^{-1}U)@>>> \Omega^p(\pi^{-1}U\times_U\pi^{-1}U)\\
@.@.@A\subset AA@A\subset AA\\
@.@.\Omega^p_\eh(Y)@>>>\Omega^p_\eh(Y\times_X Y)
\end{CD}\]
The top line is exact by the special case so that
\[ \Omega^p(\pi^{-1}U)^G\cap \Omega^p_\eh(Y)=\Omega^p_\eh(Y)^G\ \]
is the kernel $\Omega^p_\eh(X)$ of the lower line.
\end{proof}


\section{Differential forms in the $\h$-topology}\label{sec_h}
We first review the definition of the $\h$-topology and its properties from \cite{Voe96}. We then study the case of the sheaf of differential forms.

\begin{defn}[(\cite{Voe96} Definitions 3.1.1 and 3.1.2)]\label{defn_h}
A morphism $p:\tilde{X}\to X$ in $\Sch$ is called {\em topological epimorphism} if $p$ is surjective and the Zariski topology of $X$ is the quotient topology of the Zariski topology of $\tilde{X}$. It is called {\em universal topological epimorphism} if its base change by any $f:Z\to X$ in $\Sch$ is also a topological epimorphism.

The {\em $\h$-topology} on $\Sch$ is the Grothendieck topology which has as covering systems  $\{p_i:U_i\to X\}_{i\in I}$ such that 
\[ \coprod_{i\in I} p_i:\coprod_{i\in I} U_i\to X\]
is a universal topological epimorphism.
\end{defn}

\begin{ex}[(\cite{Voe96} Section 3.1)]
The following are $\h$-covers:
\begin{enumerate}
\item flat covers;
\item proper surjective morphisms;
\item quotients by the operation of a finite group.
\end{enumerate}
In particular, $\eh$-covers are $\h$-covers. The case of abstract blow-ups is particularly useful.
\end{ex}

\begin{prop}[(Blow-up square)]\label{prop_blow-up-square} 
Let $(X',Z)$ be an abstract blow-up of $X$ with $E=X'\times_XZ$. Let $\Fh$ be an $\h$-sheaf. Then the blow-up square
\[\begin{CD}
 \Fh(X')@>>>\Fh(E)\\
 @AAA@AAA\\
 \Fh(X)@>>>\Fh(Z)
 \end{CD}\]
is Cartesian.
\end{prop}
\begin{proof}
We may assume that all schemes are reduced.
Then the statement is equivalent to the sheaf condition for the $\h$-cover
$\{X'\to X, Z\to X\}$ using  the fact that $(\text{diag}:X'\to X'\times_XX', E\times_ZE\subset X'\times_XX')$ is an $\h$-cover.
\end{proof}

\begin{prop}\label{cor_normal-h}
Let 
$\mathfrak{U}=\{p_i:U_i\to X\}_{i\in I}$ be an $\h$-covering in $\Sch$. Then there exists a refinement  such that the index set $I'$ is finite and
$p'_i$ factors as
\[ U'_i\xrightarrow{\iota_i}\bar{U}'\xrightarrow{f}X'\xrightarrow{\pi}X\]
where $\{\iota_i:U'_i\to\bar{U}'\}_{i\in I}$ is an open covering in the Zariski topology, $f$ is finite surjective and $\pi$ is a proper $\eh$-covering by a smooth variety.
\end{prop}

\begin{proof}
By Construction \ref{constr_eh-locally-smooth} choose a proper $\eh$-cover $Y\to X$ with $Y$ smooth and consider the pull-back $\mathfrak{U}'$ of $\mathfrak{U}$ to $Y$. 
By \cite{Voe96} Theorem 3.1.9 we may assume that $\mathfrak{U}$ has a refinement in normal form, i.e.,
$\{U'_i\to\bar{U}'\to Y'\to Y\}_{i\in I}$ with $Y'\to Y$ a blow-up
of $Y$ in a closed subvariety, $\bar{U}'\to Y'$ finite surjective,
and the system of $\iota_i$  a finite open cover of $\bar{U}'$.
By blowing up further we may assume that $Y'$
is smooth and $Y'\to Y$ is a sequence of blow-ups in smooth centers. Then
$Y'\to Y$ is a proper $\eh$-covering by Example~\ref{lemma_propereh} (\ref{item_2.3.1}).
We choose $\pi:X'=Y'\to Y\to X$. 
\end{proof}


\begin{defn}Let $\Omega^p_h$ be the sheafification of $\Omega^p$ (or equivalently $\Omega^p_\eh$) in the $\h$-topology. Elements of $\Omega^p_h(X)$ are called
$\h$-differentials on $X$.
\end{defn}

\begin{thm}[($\h$-descent)]\label{thm_hdescent}
The presheaf $\Omega^p_\eh$ on $\Sch$ has $\h$-descent, i.e.,
\[ \Omega^p_\h(X)=\Omega^p_\eh(X)\]
for all $X\in\Sch$. In particular,
\[ \Omega^p_\h(X)=\Omega^p(X)\]
for all $X\in\Sm$.
\end{thm}

\begin{proof}
The second assertion follows from the first by
Proposition \ref{prop_smooth}.

We need to check that for  all $X$ in $\Sch$, the  morphism
\[\phi_X:\Omega^p_\eh(X) \to \lim_{\mathfrak{U}}\check{H}^0(\mathfrak{U},\Omega^p_\eh)\]
where $\mathfrak{U}$ runs through all $\h$-covers of $X$ is an isomorphism. 
Let $X$ be in $\Sch$ and $\mathfrak{U}$ an $\h$-cover of $X$. It suffices to verify the
sheaf condition for $\Omega^p_\eh$ on a refinement of $\mathfrak{U}$. We claim
that the sheaf condition is satisfied in three special cases:
\begin{enumerate}
\item when $\mathfrak{U}$ is a Zariski cover;
\item when $\mathfrak{U}=\{X'\to X\}$ is a proper $\eh$-cover;
\item when $\mathfrak{U}=\{X'\to X\}$ is a finite and surjective map of irreducible normal varieties such that $k(X')/k(X)$ is Galois.
\end{enumerate}
The first two cases hold because they are $\eh$-covers. In the third case,
 the sheaf condition holds by
Lemma \ref{lemma_Gdescent}.


Now let $\mathfrak{U}$ be a general $\h$-covering of $X$. After refining it, we may assume by Proposition~\ref{cor_normal-h}
that it is of the form $\{U_i\to\bar{U}\to X'\to X\}_{i\in I}$ with $X'\to X$ a proper $\eh$-cover with $X'$ smooth, $\bar{U}\to X$ finite and surjective and $\{U_i\}_{i\in I}$ an open cover of $\bar{U}$. For every connected component $X'_j$ of $X'$ choose an irreducible component $\bar{U}_j$ of $\bar{U}$ mapping surjectively to $X'_j$.  We refine the cover by replacing $\bar{U}$ by the disjoint union of the normalizations of
$\bar{U}_j$ in the normal hull of $k(\bar{U}_j)/k(X'_j)$. 

The sheaf condition is satisfied in the three intermediate steps.
In particular (or by Lemma \ref{cor_injdominant}) , we have injections
\[ \Omega^p_\eh(X)\hookrightarrow\Omega^p_\eh(X')\hookrightarrow\Omega^p_\eh(\bar{U})\hookrightarrow\prod_{i\in I}\Omega^p_\eh(U_i)\ .\]
Moreover, $\bar{U}\times_X\bar{U}\to X'\times_XX'$ is proper and surjective, hence by Lemma~\ref{cor_injdominant}
\[ \Omega^p_\eh(X'\times_X X')\hookrightarrow \Omega^p_\eh(\bar{U}\times_X\bar{U})\ .\]
A little diagram chase then allows us to conclude from the sheaf conditions
for $X'\to X$ and $\bar{U}\to X'$ that the sheaf condition is satisfied
for $\bar{U}\to X$. Repeating the argument with $\{U_i\to \bar{U}\}_{i\in I}$
we prove the Theorem.
\end{proof}

\begin{rem}One of the main results of Lee in \cite{Lee09} is the fact that
$\Omega^p$ is a sheaf on $\Sm$ equipped with the $h$-topology (loc. cit. Proposition 4.2). This seems basically the same as the above Theorem \ref{thm_hdescent}, though we did not check the details. His proof is different and made technically more complicated by the fact that $\Sm$ is not closed under fibre products.
\end{rem}

\begin{rem}\label{rem_simple-formula} Note that the proof contains a simple formula for $\Omega^p_h(X)$: Choose an $\h$-cover $X'\to X$ with
$X'$ smooth. Choose an $\h$-cover $X''\to X'\times_X X'$ with $X''$ smooth. Then
\[ 0\to \Omega^p_\h(X)\to\Omega^p(X')\to\Omega^p(X'')\]
is exact. 
\end{rem}

\begin{cor}\label{cor_easy-char}
Theorem \ref{thm1} of the introduction holds.
\end{cor}

\begin{proof}
We call a family $(\alpha_f)_f$ as in Theorem 1 a \emph{compatible} family. By Theorem~\ref{thm_hdescent}, any $\beta\in\Omega^p_\h(X)$  determines a compatible family $\beta_f=f^*(\beta)\in\Omega^p(Y)=\Omega^p_\h(Y)$.

Conversely, let us choose maps $X'\xrightarrow{g} X$ and $X''\xrightarrow{h} X'\times_XX'\xrightarrow{i} X$ as in Remark~\ref{rem_simple-formula}. For any compatible family $(\alpha_f)_f$ we have $(\text{pr}_1\circ h)^*\alpha_g=\alpha_{i\circ h}=(\text{pr}_2\circ h)^*\alpha_g$ so that there exists a unique $\beta\in\Omega^p_\h(X)$ such that $\alpha_g=g^*\beta$.

It remains to show that $f^*(\beta)=\alpha_f$ for any morphism $f:Y\to X$ from a non-singular variety $Y$ to $X$. Let $p:Y'\to (Y\times_XX')_\red$ be a resolution and denote the map to $X$ by $j:Y'\to X$. Then the claim follows from
\[(\text{pr}_{Y}\circ p)^*\alpha_f=\alpha_j=(\text{pr}_{X'}\circ p)^*\alpha_g=(\text{pr}_{X'}\circ p)^*g^*\beta=(\text{pr}_{Y}\circ p)^*f^*\beta\]
and the injectivity of $(\text{pr}_{Y}\circ p)^*:\Omega^p_\h(Y)\to\Omega^p_\h(Y')$.
\end{proof}

The Theorem gives a more conceptual proof for the result of Lecomte and Wach on the existence of transfers.

\begin{cor}[(\cite{LW09})]
$\Omega^p$ is an (\'etale) sheaf with transfers on $\Sm$.
\end{cor}
\begin{proof}$\Omega^p_\h$ is by definition an $\h$-sheaf. By \cite{Scholbach} Proposition 2.2 it has a canonical transfer structure. By Theorem \ref{thm_hdescent} $\Omega^p_h(X)=\Omega^p(X)$ for smooth $X$.
\end{proof}

\begin{rem}Having achieved the identification of $\Omega^p_\eh$ and $\Omega^p_\h$, the results of the previous section trivially apply to $\Omega^p_\h$ as well.
Moreover, by Corollary \ref{cor_ehcdh} we can also identify $\Omega^p_\h$ with
$\Omega^p_\cdh$. Hence all results in the series of papers
by Corti\~nas, Haesemeyer, Schlichting, Walker and Weibel
(\cite{CHSW08}, \cite{CHWei08}, \cite{CHSWaWei00}, \cite{CHWaWei10}, \cite{CHWaWei11}, \cite{CHWaWei13}) on $\Omega^p_\cdh$ can be read as results on $\Omega^p_\h$. E.g. \cite{CHSWaWei00} Theorem 4.1 in degree zero and Corollary \ref{cor_ehcdh} give a proof of our Remark \ref{rem_toric}.
\end{rem}



\section{First Properties and Examples}\label{sec:4}
In this section we gather some facts about $\h$-differential forms when considered as sheaves in the Zariski-topology. 
To this end, let us introduce the following notation.  

\begin{defn} \label{defn_restrict}
Let $\Fh$ be an $\h$-sheaf on $\Sch$. We write $\Fh|_X$ for the Zariski-sheaf $U\mapsto \Fh(U)$ for $U\subset X$ open.
\end{defn}

The basic properties of the resulting sheaves $\Omega^p_\h|_X$ and their relation to other sheaves of differential forms are summarized in the sequel.

\begin{prop}[(Properties of $\Omega^p_\h|_X$)]\label{prop_properties}
The sheaves $\Omega^p_\h|_X$ satisfy the following:
\begin{enumerate}
 \item\label{it_coh} $\Omega^p_\h|_X$ is a torsion-free coherent sheaf of $\Oh_X$-modules;
 \item \label{it_smooth} if $X$ is smooth, then $\Omega^p_\h|_X=\Omega^p_X$;
 \item\label{it_red} if $r:X_\red\to X$ is the reduction, then $r_*\Omega^p_\h|_{X_\red}=\Omega^p_\h|_X$;
 \item\label{it_kahl} if $X$ is reduced, there exists a natural inclusion  $\Omega^p_X/\torsion\subset\Omega^p_\h|_X$; 
 \item\label{it_normal} if $X$ is normal, there exists a natural inclusion $\Omega^p_\h|_X\subset \Omega^{[p]}_X$;
 \item\label{it_highdegrees} for any $p>\dim(X)$, $\Omega^p_\h|_X=0.$
\end{enumerate}
\end{prop}

\begin{proof}Recall that we have $\Omega^p_\eh=\Omega^p_\h$ by Theorem 
\ref{thm_hdescent}. 

Let us choose $X'$ and $X''$ as in Remark \ref{rem_simple-formula} and denote the maps to $X$ by $\pi':X'\to X$ and $\pi'':X''\to X$. We may assume that both $X'$ and $X''$ are proper over $X$. Then $\Omega^p_\h|_X$ is the kernel of a morphism $\pi'_*\Omega^p_{X'}\to\pi''_*\Omega^p_{X''}$ between coherent sheaves which gives the coherence in (\ref{it_coh}). It is torsion-free by Lemma~\ref{cor_injdominant}. 
Item (\ref{it_smooth}) follows from the sheafification of Theorem \ref{thm_hdescent}.
Item (\ref{it_red}) follows from the $\h$-sheaf condition for the cover
$X^\red\to X$. 

Let $X$ be normal and $U$ the smooth locus.
By Lemma \ref{cor_injdominant} and Definition \ref{defn_reflexive} the restriction map
$\Omega^p_\h(X)\to\Omega^p_\h(U)=\Omega^p(U)=\Omega^{[p]}(X)$ is injective.
Hence item (\ref{it_normal}) holds.

For Item (\ref{it_kahl}) note that by Item (\ref{it_coh}) the natural map $\Omega^p_X\to\Omega^p_\h|_X$ induces a map $\Omega^p_X/\torsion\to\Omega^p_\h|_X$. We have to check that it is injective.
Let $U\subset X_\red=X$ be the smooth locus. 
By torsion-freeness $\Omega^p_\h|_X\subset j_*\Omega^p_\h|_U=j_*\Omega^p_U$. On the other hand
$\Omega^p_{X_\red}/\torsion\subset j_*\Omega^p_{U}$ because both agree on $U$.

Finally (\ref{it_highdegrees}), we may by (\ref{it_red}) assume that $X$ is reduced, by (\ref{it_coh}) restrict to an open subset where it is smooth. By (\ref{it_smooth}) the vanishing follows from the vanishing for K\"ahler differentials.
\end{proof}

We now turn to the study of the two extreme cases $p=0$ and $p=\dim (X)$. First, let us consider the case $p=0$ and observe that $\Oh_\h=\Omega^0_\h$.

\begin{defn}[(\cite{Tra70})]\label{defn_sn}
Let $X$ be a variety with total ring of fractions $K(X)$.
The {\em semi-normalization} $X^\sn$ of $X$ is the maximal finite
cover $\pi:X^\sn\to X$ which is bijective on points
and induces an isomorphism on residue fields. If $X$ is not reduced,
we define $X^\sn$ as the semi-normalization of the reduction of $X$.

A variety is called {\em semi-normal} if it agrees with its semi-normalization.
\end{defn}

\begin{ex}
\begin{enumerate}
\item Normal varieties are semi-normal.
\item The cuspidal curve with equation $y^2=x^3$ has as semi-normalization
the affine line.
\item The nodal curve with equation $y^2=x^2(x-1)$ is semi-normal.
\end{enumerate}
\end{ex}

\begin{prop}\label{prop_sn}Let $X\in\Sch$. Then
\[ \Oh_\h(X)=\Oh( X^\sn)\ .\]
In particular, $\Oh_\h(X)=\Oh(X)$ if $X$ is semi-normal.
\end{prop}

\begin{proof}
We write $\tilde{X}$ for the presheaf $T\mapsto X(T)$ so that $\Oh=\tilde{\A}^1$. As in \cite{Voe96} Section 3.2, we write
$L(X)$ for the $\h$-sheaf associated to $\tilde{X}$ which means that $\Oh_\h=L(\A^1)$. By the universal property of sheafification
\[ \Oh_\h(X)=\Mor_{L(\Sch)}(L(X),L(\A^1))\]
where $L(\Sch)$ is the category of representable h-sheaves. Moreover,
\cite{Voe96} Proposition 3.2.10 asserts that 
\[ \Oh_\h(X)=\Mor_{L(\Sch)}(L(X),L(\A^1))=\Mor_{\Sch}(RL(X),\A^1)=\Oh(RL(X))\]
where $RL(X)=X^\sn$ is the semi-normalization of $X$.
\end{proof}

We now turn to the other extreme case $p=\dim X$.

\begin{prop}\label{prop_omega-top-degree} Let $X$ be a 
variety of dimension $d$ and $\pi:\tilde{X}\to X$
a resolution, i.e., a proper birational morphism with
$\tilde{X}$ smooth. Then
\[ \Omega^{d}_\h|_X=\pi_*\Omega^d_{\tilde{X}}. \]
\end{prop}

\begin{proof} Let $E\subset\tilde{X}$ be the reduced exceptional locus of $\pi$, and let $i:Z:=\pi(E)\to X$ be the inclusion map of its image in $X$. The blow-up sequence 
\[ 0 \to \Omega^d_\h|_X\to \pi_*\Omega^d_\h|_{\tilde{X}}\oplus i_*\Omega^d_\h|_Z \to i_*\pi|_{E*}\left(\Omega^d_\h|_E\right) \]
obtained from Proposition \ref{prop_blow-up-square} together with the vanishing $\Omega^d_\h|_E=0$ and $\Omega^d_\h|_Z=0$ by Proposition \ref{prop_properties}(\ref{it_highdegrees}) complete the proof. 
\end{proof}

Recall that any projective scheme has a dualizing sheaf in the sense of \cite{Har77} Section III.7.

\begin{cor}
Let $X$ be a normal projective variety of pure dimension $d$ with dualizing sheaf $\omega^o_X$. Then there exists an inclusion
\[ \Omega^d_\h|_X\subset \omega^o_X .\]
If in addition $X$ has only rational singularities (see before Remark \ref{prop_rational}), then
\[ \Omega^d_\h|_X = \omega^o_X.\]
\end{cor}
\begin{proof}The dualizing sheaf is always reflexive, hence $\omega^o_X=j_*\Omega^d_{X^\reg}$ where $j:X^\reg\to X$ is the inclusion of the regular locus. Hence
the first statement is simply Proposition \ref{prop_properties} (\ref{it_normal}).

Now assume  that $X$ has rational singularities. 
Let $\pi:\tilde{X}\to X$ be a resolution. 
 Let $\omega_X$ be the dualizing complex of $X$ normalized such that
$\pi^!\omega_X=\Omega^d_{\tilde{X}}[d]$. By the definition of rational singularities
\[ \Oh_X\to R\pi_*\Oh_{\tilde{X}}\]
is a quasi-isomorphism. 
By Grothendieck duality, this implies that
\[ R\pi_*\Omega^d_{\tilde{X}}[d]\to \omega_X\]
is a quasi-isomorphism. Taking cohomology $H^{-d}$ in degree $-d$ yields $\omega^0_X=\pi_*\Omega^d_{\tilde{X}}$. (Note that the dualizing sheaf $\omega^o$ is the first non-vanishing cohomology sheaf $H^{-d}\omega_X$ of the dualizing complex.)
\end{proof}

\begin{rem}The above proof does not use the fact that a variety with rational singularities is Cohen-Macaulay. Hence it avoids the use of Kodaira vanishing.
\end{rem}

We now turn to varieties with special types of singularities. We say that a scheme $X$ has \emph{normal crossings} if $X$ is Zariski-locally isomorphic to a normal crossings divisor in a smooth variety or equivalently, if $X$ is \'etale locally isomorphic to a union of coordinate hyperplanes in the affine space. Observe that this implies that the irreducible components of some \'etale cover are smooth.

\begin{prop}[(Normal crossing schemes)]\label{prop_h-dnc}
Let $X$ be a scheme with normal crossings. Then
\[ \Omega^p_h|_{X}=\Omega^p_X/\torsion \ .\]
\end{prop}

\begin{proof} There is a natural inclusion $\Omega^p_X/\torsion\subset\Omega^p_\h|_X$ by Proposition \ref{prop_properties}(\ref{it_kahl}). We may work \'etale locally in order to show that it is an isomorphism. Hence we can assume that $X$ is a union of coordinate hyperplanes in the affine space.

We prove the claim by induction on the number $c(X)$ of irreducible components of $X$, the case $c(X)\leq 1$ following from Proposition \ref{prop_properties} (\ref{it_smooth}). For $c(X)>1$ choose some irreducible component $E\subset X$ and let $X'=\overline{X\backslash E}.$ The blow-up sequence associated with $(X'\to X, E\to X)$ fits into a commutative diagram
\[
\xymatrix{0 \ar[r] & \Omega^p_\h|_X \ar[r] & \Omega^p_\h|_{X'}\oplus\Omega^p_\h|_E \ar[r] & \Omega^p_\h|_{X'\cap E} \\
 & \Omega^p_X/\text{tor} \ar[r] \ar@{^{(}->}[u]_{\text{inj. by}\atop\ref{prop_properties}(\ref{it_kahl})} &\Omega^p_{X'}/\text{tor}\oplus\Omega^p_E/\text{tor} \ar[r]\ar@{=}[u]^{\text{induction}\atop\text{on } c}_{\text{smooth case}\atop\ref{prop_properties}(\ref{it_smooth})} &\Omega^p_{X'\cap E}/\text{tor}\ar@{^{(}->}[u]_{\text{inj. by}\atop\ref{prop_properties}(\ref{it_kahl})}}\]
of sheaves on $X$, where the horizontal maps in the second row are given by the pullback of torsion-free differential forms constructed in \cite{Fer70} Proposition 1.1. By a diagram chase one reduces the proof to showing that the second row is exact. This can be checked by a calculation using local equations for $X$ in a smooth ambient space.
\end{proof}

\begin{prop}[(Quotient singularities)]
Let $X$ be a smooth quasi-projective variety with an operation of a finite group $G$. Then
\[ \Omega^p_h|_{X/G}=(\Omega_X^p)^G=\Omega^{[p]}_{X/G}\ .\]
\end{prop}
\begin{proof} Recall that $X/G$ is normal since $X$ is so and locally $X/G=\text{Spec}(A^G)$ where $X=\text{Spec}(A)$. Moreover $X\to X/G$ is a ramified cover
with Galois group $G$. The first assertion is immediate from Lemma \ref{lemma_Gdescent}. The second
was established by Knighten for $p=1$ in \cite{Kni73} and for general $p$ by Lecomte and Wach in \cite{LW09}.
\end{proof}

The case of klt singularities will be treated in Section \ref{sec:5}, see Theorem \ref{thm_klt-sing}.


\section{Application: Reflexive forms on klt base spaces}\label{sec:5}
In this section we examine the sheaf of $\h$-differential forms on a complex variety whose singularities are mild in the sense of the Minimal Model Program. More precisely, we are concerned with the following class of singular varieties:

\begin{defn}\label{defn_klt-base-space}
An irreducible variety $X$ over the complex numbers is said to be a \emph{klt base space}, if there exists an effective $\mathbb{Q}$-divisor $\Delta\geq 0$ such that the pair $(X,\Delta)$ has Kawamata log terminal singularities (see \cite{KM98} Definition 2.34).
\end{defn}

\begin{rem}
In Definition~\ref{defn_klt-base-space} we follow the terminology in \cite{Keb12} Definition~5.1. This notion is equivalent to the definition of log terminal singularities in the sense of \cite{dFH09} Theorem 1.2.
\end{rem}

\begin{ex}\label{ex_toric}
A normal toric variety $X$ is locally a klt base space. Indeed, a normal projective toric variety is a klt base space by \cite{CLS11} Example 11.4.26.
\end{ex}

Recall that a klt base space $X$ is normal by definition so that $\Omega^p_\h|_X\subset\Omega^{[p]}_X$ by Proposition \ref{prop_properties} (\ref{it_normal}). The following theorem establishes the inverse inclusion and will be proved at the end of the section.

\begin{thm}\label{thm_klt-sing}
Let $X$ be a klt base space and $p\geq 0$. Then
\[ \Omega^p_\h|_X=\Omega^{[p]}_X\ . \]
\end{thm}

\begin{rem}\label{rem_toric}
Theorem~\ref{thm_klt-sing} applies in particular to normal toric varieties by Example~\ref{ex_toric}. In view of Corollary~\ref{cor_ehcdh} this case has already been proved in \cite{CHSWaWei00} Theorem 4.1.
\end{rem}

In order to deduce Theorem \ref{thm_klt-sing} from results on reflexive differential forms obtained by Greb, Kebekus, Kov\'acs and Peternell we need a result on the pullback of $\h$-differential forms under morphisms with rationally chain connected fibers proved in the next subsection, which may be of independent interest.

As an application we find a more conceptual proof of a recent result of Kebekus.

\begin{cor}[(\cite{Keb12})]\label{cor-keb-pullback}
If $f:X\to Y$ is a morphism between klt base spaces, then there is a natural pullback map
\[\Omega^{[p]}(Y)\to\Omega^{[p]}(X)\]
compatible with pullback in the smooth case. I.e.,
$\Omega^{[p]}$ is a presheaf on the category of klt base spaces.
\end{cor}

The assumption on the type of singularities of $X$ in Theorem \ref{thm_klt-sing} seems to be optimal, since even for log canonical singularities the theorem fails.

\begin{ex}[(\cite{Keb12} Section 1.2)]\label{ex_notrefl}
Let $E\subset\mathbb{P}_{\mathbb{C}}^2$ be an elliptic curve and let $X\subset\mathbb{A}^3$ be the affine cone over $E$, with vertex $p\in X$ and projection map $\pi:X\smallsetminus\{p\}\to E$. Recall that $X$ has log canonical singularities. We claim that the inclusion
\[ \Omega^1_\h|_X\subsetneq\Omega^{[1]}_X\]
is strict. For a non-zero $\alpha\in H^0(E,\Omega^1_E)$ we get a non-zero element $\pi^*(\alpha)\in H^0(X\smallsetminus \{p\},\Omega^1)=\Omega^{[1]}(X)$. 
Suppose that $\pi^*(\alpha)=\alpha'\in\Omega^1_\h(X)\subset \Omega^{[1]}(X)$. Let $\phi:\tilde{X}\to X$ denote the blow-up of the vertex with exceptional set $\mathrm{Exc}(\phi)\cong E$. It is an $\A^1$-bundle over $E$ with $\alpha'$ the pull-back of $\alpha$.
 We find a contradiction by calculating
\[\alpha = \left(\phi^*\alpha'\right)|_{\mathrm{Exc}(\phi)} = \phi|_{\mathrm{Exc}(\phi)}^*\left(\alpha'|_{\{p\}}\right)=0. \]
\end{ex}

\subsection{Rationally chain connected fibrations} In this section, $k$ is an arbitrary field of characteristic $0$.

\begin{defn}\label{defn_rel-rcc}
Suppose that $k$ be algebraically closed. We say that a $k$-scheme $X$ is rationally chain connected, if there exists a family $g:U\to B$ of proper curves together with a morphism $U\to X$ such that
\begin{enumerate}
 \item the fiber $U_b$ over any $k$-valued point $b$ of $B$ is connected and has only rational irreducible components; and
 \item the morphism $U\times_BU\to X\times X$ is dominant.
\end{enumerate}
\end{defn}

\begin{lemma}[(\cite{Kol96} Corollary IV.3.5.1)]\label{lem_rcc-get-all-points}
Suppose that $k$ is algebraically closed. Let $X$ be a proper and rationally chain connected scheme. Then for arbitrary closed points $x_1,x_2\in X$ there exists a proper connected curve $C$ with rational irreducible components together with a morphism $C\to X$ whose image contains $x_1$ and $x_2$.
\end{lemma}

\begin{lemma}\label{lemma_pullback-rcc-product}
Suppose that $k$ be algebraically closed. Let $C$ be a proper one-dimensional connected scheme all of whose irreducible components are rational curves, and let $X$ be an arbitrary scheme. Then the pullback by the first projection $p_X:X\times C\to X$ yields a bijection
\[ \Omega^p_\h(X) = \Omega^p_\h(X\times C)\]
for any $p\geq 0$.
\end{lemma}

\begin{proof}It suffices to check the assertion locally in the $h$-topology. In particular we may assume that $X$ is smooth. In the case $C=\mathbb{P}^1$ the assertion is an immediate consequence of Theorem~\ref{thm_hdescent} and the formula 
\[ \Omega^p_{X\times\mathbb{P}^1}=\text{pr}_X^*(\Omega^p_X)\oplus \bigl(\text{pr}_X^*(\Omega^{p-1}_X)\otimes \text{pr}_{\mathbb{P}^1}^*(\Omega^1_{\mathbb{P}^1})\bigr), \]
since the right hand summand has no non-zero global section.

In the general case, we may assume that $C$ is reduced. Let $C^\nu$ be the normalization of $C$ and let $C_{\sing}$ be the singular locus with its reduced scheme structure. Then $(C^\nu\to C, C_{\sing}\to C)$ is an abstract blow-up. We denote by $E=C^{\nu}\times_CC_{\sing}$ the inverse image of $C_{\sing}$ in $C^{\nu}$. The case $C=\mathbb{P}^1$ already treated above shows that the Cartesian blow-up square given by Proposition~\ref{prop_blow-up-square} is
\[\begin{CD}
 \prod_{D \in \mathrm{Comp}(C^{\nu})}\Omega^p(X)@>>>\prod_{e\in E}\Omega^p(X)\\
  @AAA@AAA\\
 \Omega^p_\h(X\times C)@>>>\prod_{x\in C_{\sing}}\Omega^p(X).
 \end{CD}\]
and the assertion follows using the connectedness of $C$.
\end{proof}

\begin{lemma}\label{lemma-rcfamily}
Suppose that $k$ is uncountable and algebraically closed. Let $f:X\to Y$ be a projective morphism between $k$-schemes such that
\begin{enumerate}
 \item the fiber $X_y$ over any closed point $y\in Y$ is a rationally chain connected $k$-scheme; and
 \item there exists a section $s:Y\to X$, i.e., $f\circ s=\id_Y$.
\end{enumerate}
Then there exists a one-dimensional proper connected scheme $C$ with rational irreducible components and a scheme $H$ together with a dominant morphism
\[ \psi:H\times Y \times C \to X \]
over $Y$ such that for some closed point $c\in C$ the map $\psi$ satisfies $\psi(b,y,c)=s(y)$ for all closed points $b\in B, y\in Y$.
\end{lemma}

\begin{proof}
First we introduce notation: For integers $r\geq 0, 1\leq i\leq r$, let $l_i\subset\mathbb{P}^r$ be the line defined by the equations $x_j=0$ for $0\leq j\leq r$ and $j\neq i,i-1$. We denote by $C_r\subset\mathbb{P}^r$ the union of all $l_i$, $1\leq i\leq r$ with its reduced scheme structure. Finally, let $c_r:=[1:0:\cdots :0]\in C_r$.

By Lemma~\ref{lem_rcc-get-all-points} any two closed points of $X$ lying in the same fiber $f^{-1}(y)$ over some closed point $y\in Y$ can be connected by a proper curve $C\subset X$ with rational irreducible components, and we may even assume $C\cong C_r$ for some $r\geq 0$. 

Let $H_r:=\mathit{Hom}_{Y,s}(Y\times C_r,X)$ be the relative $\mathit{Hom}$-scheme parameterizing morphisms $Y\times C_r\to X$ over $Y$ whose restriction to $Y\cong Y\times \{c_r\}\subset Y\times C_r$ coincides with $s$. Recall that $H_r$ is a countable union of schemes of finite type over $k$ by \cite{Kol96} Theorem 1.10. 

The above considerations show that the universal morphism
\[ \phi=\bigcup_r\phi_r:\bigcup_{r\geq 0}H_r\times Y\times C_r\to Y \]
is surjective on $k$-points. Since the field $k$ is uncountable and the image of $\phi_r$ is contained in that of $\phi_{r+1}$, there exists an $r\geq 0$ and an open subscheme $H\subset H_r$ of finite type over $k$ such that the morphism $H\times Y\times C_r\to Y$ is dominant, see \cite{Har77} Exercise V.4.15(c).
\end{proof}

\begin{thm}\label{thm_rcc}
Let $f:X\to Y$ be a surjective projective morphism between $k$-schemes such that the fiber over any $\overline{k}$-valued point of $Y$ is a rationally chain connected $\overline{k}$-scheme. Then the pullback map 
\[ \Omega^p_\h(Y)\xrightarrow{\isom}\Omega^p_\h(X) \] 
yields a bijection
for any $p\geq 0$.
\end{thm}

\begin{proof}
By base change we may assume that the field $k$ is algebraically closed and uncountable. 

Let us first reduce the proof to the case when $f:X\to Y$ admits a section $s:Y\to X$. To this end, we perform a base change by $f:X\to Y$. Writing $W=X\times_YX$ we find a commutative diagram
\[
\xymatrix{0 \ar[r] & \Omega^p_\h(X) \ar[r] & \Omega^p_\h(X\times_Y X) \ar[r] & \Omega^p_\h(X\times_YW) \\
0 \ar[r] & \Omega^p_\h(Y)\ar[r] \ar@{^{(}->}[u] & \Omega^p_\h(X) \ar[r]\ar@{^{(}->}[u] &\Omega^p_\h(W) \ar@{^{(}->}[u]}\]
whose rows are exact and whose vertical maps are injective because
they are induced hy $\h$-covers.
In particular the theorem holds for $f:X\to Y$ if it holds for the first projection $\text{pr}_1:X\times_YX\to X$, which admits a section.

From now on we assume that $f:X\to Y$  admits a section $s:Y\to X$. Let $\psi:H\times Y\times C\to X$ and $v\in V$ as in Lemma~\ref{lemma-rcfamily}. In particular we have a commutative diagram
\[\begin{CD}
 H\times Y\times C@>\psi >>X\\
  @A incl AA@AA s A \\
 H\times Y\times\{c\}@>> p_Y >Y
 \end{CD}\]
which yields a commutative diagram
\[\begin{CD}
 \Omega^p_\h(H\times Y\times C)@<\psi^* <<\Omega^p_\h(X)\\
  @V incl^* VV@VV s^* V \\
 \Omega^p_\h(H\times Y\times\{c\})@<< p_Y^* <\Omega^p_\h(Y).
 \end{CD}\]
Recall that by Lemma~\ref{cor_injdominant} the pullback $\psi^*$ under the dominant map $\psi$ is injective. Moreover $incl^*$ is bijective by Lemma~\ref{lemma_pullback-rcc-product}. We deduce that the pullback $s^*$ under the section $s$ is injective. By
\[ \id:\Omega^p_\h(Y)\xrightarrow{f^*}\Omega^p_\h(X)\xrightarrow{s^*}\Omega^p_\h(Y)\]
this proves the assertion.
\end{proof}

\begin{rem}
If $f$ and $Y$ are in addition smooth the theorem is  classical. In the case $k=\mathbb{C}$, $Y=pt$ and $X$ an irreducible projective variety with Kawamata log terminal singularities, the analogous result was proved in \cite{GKKP11} Theorem 5.1 for reflexive differentials.
\end{rem}

\subsection{Proof of Theorem \ref{thm_klt-sing}}

Let $\pi:\tilde{X}\to X$ be a resolution of singularities. The main result of \cite{GKKP11} states that $\Omega^{[p]}(X)=\Omega^p(\tilde{X})$. On the other hand, \cite{HM07} Corollary 1.5 asserts that all fibers of $\pi$ are rationally chain connected so that $\Omega^p_\h(X)=\Omega^p_\h(\tilde{X})=\Omega^p(\tilde{X})$ by Theorem \ref{thm_rcc}.

\section{Cohomology of differential forms}\label{sec_coho}

We now study cohomology of sheaves of differential forms in the $\h$-topology. We will first assemble some technical tools on the cohomology of $\h$-sheaves and then apply them to the case of differential forms.

\subsection{Cohomology of $\h$-sheaves}

\begin{prop}\label{prop_cohoehh}
Let $\Fh$ be a sheaf of $\mathbb{Q}$-vector spaces on $Sch_\h$, $X\in\Sch$ and $p\geq 0$. Then
\[ H^p_\h(X,\Fh) = H^p_\eh(X,\Fh)\ . \]
\end{prop}

\begin{proof}Let $\eta:\Sch_\h\to\Sch_\eh$ be the morphism of sites. 
We need to show $R^p\eta_*\Fh=0$ for $p\geq 1$. 

There exists a short exact sequence $0\to\Fh\to\mathcal{I}\to\mathcal{G}\to 0$ of $\h$-sheaves of $\Q$-vector spaces where $\mathcal{I}$ is an injective object in the category of $\h$-sheaves of $\Q$-vector spaces.  It is automatically injective as sheaf of abelian groups.
Since $R^{p+1}\eta_*\Fh\cong R^p\eta_*\mathcal{G}$ for $p\geq 1$ it suffices to show that $R^1\eta_*\Fh=0$ for all $\Fh$.

$R^1\eta_*\Fh$ is the $\eh$-sheafification of the presheaf $X\mapsto H^1_\h(X,\Fh)$. We only need to show that $H^1_\h(X,\Fh)=H^1_\eh(X,\Fh)$.
If $t$ is a Grothendieck topology on $\Sch$, then
by \cite{Milne} Corollary III.2.10 
\[ H^1_t(X,\Fh)=\lim_{\mathfrak{U}} \check{H}^1(\mathfrak{U}, \Fh)\ .\]
where $\mathfrak{U}$ runs through the system of $t$-covers of $X$. We
apply this to $t$ equal to $\h$, $\eh$, $\et$ (\'etale topology) and 
$\qfh$ (quasi-finite $\h$; covers are $\h$-covers which are quasi-finite).
We introduce the notation $t(X)$ for the system of $t$-covers of $X$.
By Proposition \ref{cor_normal-h}, any $h$-cover $\mathfrak{U}$ can
be refined by a composition of a $\qfh$-cover followed by an $\eh$-cover
$X'\to X$ with $X'$ smooth. Hence
\[H^1_\h(X,\Fh)=\lim_{\mathfrak{U}\in\h(X)}\check{H}^1(\mathfrak{U}, \Fh)
           =\lim_{X'\in\eh(X)}\lim_{\mathfrak{U}\in\qfh(X')}\check{H}^1(\mathfrak{U}, \Fh)
\]
with $X'$ smooth. By \cite{Voe96} Theorem 3.4.1 we have
\[ 
	\lim_{\mathfrak{U}\in\qfh(X')}\check{H}^1(\mathfrak{U}, \Fh)
		=H^1_\qfh(X',\Fh)=H^1_\et(X',\Fh)
		=  \lim_{\mathfrak{U}\in\et(X')}\check{H}^1(\mathfrak{U}, \Fh)
\]
because $X'$ is smooth and $\Fh$ a sheaf of $\Q$-vector spaces. This implies
\[ H^1_\h(X,\Fh)=\lim_{X'\in\eh(X)}\lim_{\mathfrak{U}\in\et(X')}\check{H}^1(\mathfrak{U}, \Fh)=\lim_{\mathfrak{U}\in\eh(X)}\check{H}^1(\mathfrak{U}, \Fh)
=H^1_\eh(X,\Fh) \]
because \'etale covers are $\eh$-covers.
\end{proof}

\begin{prop}[(Blow-up sequence)]\label{prop_blowup-cohom}
Let $(X',Z)$ be an abstract blow-up of $X$ with $E=X'\times_XZ$. Let $\Fh$ be an $\h$-sheaf. Then the blow-up sequence
\[\dots\to H^i_\h(X,\Fh)\to H^i_\h(X',\Fh)\oplus H^i_\h(Z,\Fh)\to H^i_\h(E,\Fh)\to H^{i+1}_\h(X,\Fh)\to \dots\]
is exact.
\end{prop}

\begin{proof} The argument is given by Geisser in \cite{Gei06} Proposition 3.2 for the $\eh$-topology and can be applied to the $\h$-topology without changes.
\end{proof}

\begin{prop}\label{prop_sncoh}
Let $X\in \Sch$ with semi-normalization $X^\sn$ (see Definition \ref{defn_sn}) and let $\Fh$ be an $\h$-sheaf of abelian groups. Then
\[ H^i_\h(X,\Fh)=H^i_\h(X^\sn,\Fh)\ .\]
\end{prop}

\begin{proof}
We take up the notation of the proof of Proposition \ref{prop_sn}, i.e,
let $\tilde{X}$ be the presheaf $T\mapsto X(T)$ and $L(X)$ its
$\h$-sheafification. By \cite{Voe96} Theorem 3.2.9, the semi-normalization
induces an isomorphism of sheaves of sets
\[ L(X^\sn)\isom L(X)\ .\]
This implies an isomorphism of $\h$-sheaves of abelian groups
\[\Z_\h(X^\sn)\to \Z_\h(X)\ .\]
Hence
\[ H^i_\h(X,\Fh)=\Ext^i(\Z_\h(X),\Fh)\isom\Ext^i(\Z_\h(X^\sn),\Fh)=H^i_\h(X^\sn,\Fh)\ . \]
\end{proof}

\subsection{Cohomology of $\h$-differential forms}

\begin{cor}\label{lem_sncoh}Let $X\in \Sch$ and $X^\sn$ be the semi-normalization (see Definition \ref{defn_sn}) of $X$. Then
\[ H^i_\h(X,\Omega^p_\h)=H^i_\h(X^\sn,\Omega^p_\h)\ .\]
\end{cor}
\begin{proof} This is a special case of Proposition \ref{prop_sncoh}.
\end{proof}

\begin{cor}[Smooth varieties]\label{cor_ehcohom}
If $X\in\Sm$, then
\[ H^i_\h(X,\Omega^p_\h) = H^i_\eh(X,\Omega^p_\eh) = H^i_\Zar(X,\Omega^p)\ .\]
\end{cor}
\begin{proof} This follows from Proposition \ref{prop_cohoehh} and \cite{Gei06} Theorem 4.7.
\end{proof}
\begin{rem}This was first proved with a different argument in a version for $\Sm$ equipped with the $\h$-topology by Lee, see \cite{Lee09} Prop. 4.2. 
\end{rem}

\begin{cor}[(Cohomological dimension)]\label{cor_cohdim}Let $X\in\Sch$. Then
\[ H^i_\h(X,\Omega^p_\h)=0\hspace{2ex}\text{for $i>\dim X$.}\]
\end{cor}
\begin{proof} We argue by induction on the dimension of $X$, the case $\dim(X)<0$ being trivial. By Corollary \ref{lem_sncoh} we may assume that $X$ is reduced so that there exists an abstract blow-up $(X',Z)$ such that $X'$ is smooth and $\dim(Z)<\dim(X)>\dim(Z\times_XX')$. Using the blow-up sequence \ref{prop_blowup-cohom} the claim for $X$ follows from that for the smooth scheme $X'$ done by Corollary \ref{cor_ehcohom} and the inductive hypothesis applied to $Z$ and $X'\times_X Z$.
\end{proof}

\begin{cor}[(Vanishing)]\label{cor_vanishing}
Let $X$ be a variety of dimension $d$. Then
\[ H^i_\h(X,\Omega^p_\h)=0\hspace{1cm}p>d\ .\]
\end{cor}
\begin{proof}The argument is the same as for the cohomological dimension.
\end{proof}

\begin{cor}[(Finiteness)]Let $X$ be proper. Then $H^i_\h(X,\Omega^p_\h)$ is
finite dimensional.
\end{cor}
\begin{proof}The assertion holds for smooth proper $X$ because $\Omega^p_X$
is coherent. The general case follows by induction on the dimension from
the blow-up sequence.
\end{proof}
Recall the notion of an $\h$-hypercover, e.g. from \cite{HodgeIII} Definition 5.3.4.
It is a morphism of simplicial schemes 
\[ X_\bullet\to Y_\bullet\]
satisfying certain conditions which ensure that
\[ H^i_\h(Y_\bullet,\Fh)\to H^i_\h(X_\bullet,\Fh)\]
is an isomorphism for all $\h$-sheaves $\Fh$.

\begin{prop}\label{prop_hypercov}Let $X\in\Sch$ and $X_\bullet\to X$ an $\h$-hypercover such
that all $X_n$ are smooth. Then
\[ H^i_\h(X,\Omega^p_\h)\isom H^i_\Zar(X_\bullet,\Omega^p_{X_\bullet})\ .\]
There is a natural spectral sequence
\[ E_1^{nm}=H^m_\Zar(X_n,\Omega^p_{X_n})\Rightarrow H^{n+m}_\h(X,\Omega^p_\h)\ .\]
\end{prop}
\begin{proof} This follows from Corollary \ref{cor_ehcohom} and the general
descent formalism as explained e.g. in \cite{HodgeIII} Section 5.3.
\end{proof}

The blow-up sequence easily allows the computation of cohomology with coefficients in the canonical sheaf.

\begin{ex}\label{ex_topdegree}Let $X$ be a variety of dimension $d$. Let $X_1,\dots,X_n$ be the irreducible components of $X$ which have dimension $d$. Let $\tilde{X}_j\to X_j$ be a resolution. Then
\[ H^i_\h(X,\Omega^d_\h)=\bigoplus_{j=1}^nH^i_\Zar(\tilde{X}_j,\Omega_{\tilde{X}_i}^d)\ .\]
\end{ex}
\begin{proof}Let $X_{n+1},\dots,X_{N}$ be the irreducible components of $X$ whose dimension is strictly less than $d$. Let $X'=\coprod_{j=1}^NX_j$. Then there is an abstract blow-up $(\pi:X'\to X,Z)$ with $Z$ the locus where irreducible components intersect.
Both $Z$ and its preimage in $X'$ have dimension strictly smaller than $d$. 
By Corollary \ref{cor_vanishing} its cohomology with coefficients in $\Omega^d_\h$ vanishes. By the blow-up sequence this implies 
\[ H^i_\h(X,\Omega^d_\h)=H^i_\h(X',\Omega^d_\h)=\bigoplus_{j=1}^NH^i_\h(X_j,\Omega^d_\h)\ .\]
The components of dimension smaller than $d$ do not contribute by the same argument. Finally, let $Z_j\subset X_j$ be the locus where $\tilde{X}_j\to X_j$ is not an isomorphism. We use the blow-up sequence for $(\tilde{X}_j\to X_j,Z_j)$ to conclude.
\end{proof}

\subsection{The derived push-forward}\label{sec_rho}
After studying the $\h$-cohomology of $\h$-differential forms in the preceding subsection we systematically compare h-cohomology with Zariski cohomology.
From now on, we work systematically in the derived category of abelian Zariski-sheaves.

\begin{defn}Let
\[ \rho:\Sch_\h\to\Sch_\Zar\]
be the canonical morphism of sites. 
For $X\in\Sch$ we denote
\[ \rho_X:\Sch_\h\to X_\Zar\]
the inclusion of $X$ with the Zariski topology.
By abuse of notation, we also denote
\[ \rho:(\Sch_\h)^\sim\to(\Sch_\Zar)^\sim\ ,
   \hspace{2ex}\rho_X:(\Sch_\h)^\sim\to X_\Zar^\sim
 \]
the induced morphism of topoi.
\end{defn}
\begin{rem}There are also versions with the \'etale topology instead of
the Zariski topology. For our purposes it does not make a difference which
to use.
\end{rem}
Note that $\rho^*$ is nothing but the h-sheafification. It is exact. The functor $\rho_*$ is left exact.
We are going to
consider its right derived functor. In accordance with Definition \ref{defn_restrict} we write suggestively
\[ \Fh|_X=\rho_{X*}\Fh\ .\]

\begin{prop}\label{prop_smoothdb} Let $X$ be smooth. Then the adjunction map
\[ \Omega_X^p\to R\rho_{X*}\Omega^p_\h\]
is a quasi-isomorphism, i.e.,
\[ R^i\rho_{X*}\Omega_\h^p=\begin{cases}\Omega_X^p&i=0\\
                                        0&i>0.
			  \end{cases}\]
\end{prop}

\begin{proof}
This is a reformulation of Corollary \ref{cor_ehcohom}.
\end{proof}

As in the absolute case this also allows us to compute $\rho_{X*}$ for singular spaces. Our main tool is again the blow-up sequence.

\begin{prop}[(Blow-up triangle)]\label{prop_blowuptriangle}
Let $(\pi:X'\to X,i:Z\to X)$ be an abstract blow-up of $X$ with $\pi':E=X'\times_XZ\to Z$. Let $\Fh$ be an h-sheaf. Then there is a natural distinguished triangle
\[ R\rho_{X*}\Fh\to R\pi_* R\rho_{X'*}\Fh\oplus i_* R\rho_{Z*}\Fh\to
 i_*R\pi'_{E*}R\rho_{E*}\Fh\xrightarrow{[1]}\ .\]
\end{prop}
\begin{proof}
As Proposition \ref{prop_blowup-cohom} this can be proved using the same arguments as in \cite{Gei06} Proposition 3.2.
\end{proof}

\begin{cor}\label{prop_rhocohdim}Let $X\in\Sch$ and $\pi:X_\bullet\to X$ a proper hypercover
with $X_n$ smooth for all $n$. Then
\[ R\rho_{X*}\Omega^p_\h=R\pi_*\Omega^p_{X_\bullet}\ .
\]
Moreover:
\begin{enumerate}
\item The complex is concentrated in degrees at most $\dim X$. 
\item It vanishes for $p>\dim X$.
\item All cohomology sheaves are coherent.
\end{enumerate}
\end{cor}
\begin{proof}
This is the sheafification of Proposition \ref{prop_hypercov}.
The sheafification of Corollary~\ref{cor_cohdim} gives (1). Using an induction on dimension and the blow-up triangle in Proposition \ref{prop_blowuptriangle}, one can deduce (2) easily from Proposition \ref{prop_smoothdb}. Assertion (3) follows because
$\Omega^p_{X_\bullet}$ is coherent and $\pi$ proper.
\end{proof}
\begin{rem}By construction this says that the (shifted) $p$-th graded piece $\uOmega^p_X$ of the Du Bois complex is nothing but $R\rho_{X*}\Omega^p_\h$. See Section \ref{sec_dubois} for a more detailed discussion. In particular, all statements about
$R\rho_{X*}\Omega^p_\h$ in this section can be read as statements about
$\uOmega^p_X$.
\end{rem}

The methods also allow us to compute explicit cases.

\begin{cor}\label{cor_dnc} Let $X$ be a  normal crossing scheme (see Proposition \ref{prop_h-dnc}) . Then 
\[ \Omega^p_\h|_X\to R\rho_{X*}\Omega^p_\h\]
is a quasi-isomorphism. In particular,
\[ H^i_\h(X,\Omega^p_\h)\isom H^i_\Zar(X,\Omega^p_\h|_X)=H^i_\Zar(X,\Omega^p_X/\torsion)\ .\]
\end{cor}

\begin{proof} The assertion can be checked \'etale locally. Hence we may assume that $X$ is a union of coordinate hyperplanes in the affine space.

We prove the claim by induction on the number of irreducible components of $X$, the case $X=\emptyset$ being trivial. For an arbitrary irreducible component $E\subset X$ both $X':=\overline{X\backslash E}$ and $E':=E\cap X'$ are unions of coordinate hyperplanes in affine spaces with fewer irreducible components than $X$. Write $i_\bullet:\bullet\to X$ for the inclusion of a closed subscheme $\bullet$, so that the proof of Proposition \ref{prop_h-dnc} and the blow-up triangle yield a morphism of distinguished triangles
\[
\xymatrix{
\Omega^p_\h|_X \ar[r]\ar[d] & i_{X'*}\Omega^p_\h|_{X'}\oplus i_{E*}\Omega^p_\h|_{E} \ar[r]\ar[d] & i_{E'*}\Omega^p_\h|_{E'} \ar[r]^-{\text{+1}}\ar[d] & \\
R\rho_{X*}\Omega^p_\h \ar[r] & i_{X'*}R\rho_{X'*}\Omega^p_\h\oplus i_{E*}R\rho_{E*}\Omega^p_\h \ar[r] & i_{E'*}R\rho_{E'*}\Omega^p_\h \ar[r]^-{+1} &. 
}\]
The claim follows from the inductive hypothesis applied to $X'$ and $E'$ and Proposition \ref{prop_smoothdb} applied to $E$.

The second identification follows from Proposition \ref{prop_h-dnc}.
\end{proof}

The case of top degree  is particularly interesting. 

\begin{prop}\label{prop_topdb} Let $X$ be a variety of dimension $d$. Let $\pi:\tilde{X}\to X$
be a resolution, i.e., proper birational with $\tilde{X}$ smooth. Then
\[ R\rho_{X*}\Omega^d_\h=R\pi_*\Omega^d_{\tilde{X}}\ .\]
\end{prop}

\begin{proof}
The transformation $\rho_{X*}\to \pi_*\circ \rho_{\tilde{X}*}$ induces by Proposition \ref{prop_smoothdb} a map 
\[R\pi_*\Omega^d_{\tilde{X}}=R\pi_*R\rho_{\tilde{X}*}\Omega^d_\h\to R\rho_{X*}\Omega^d_\h,\]
which is a quasi-isomorphism by the sheafification of Example \ref{ex_topdegree}.
\end{proof}

Note that this formula is independent of the choice of resolution. This allows us to give an easy proof of a statement which does not involve $\h$-differentials at all. 

\begin{cor}\label{cor_birationalomega} Let $\pi:\tilde{X}\to X$ be a proper birational morphism between 
smooth connected varieties of dimension $d$. Then
\[ \Omega^d_X\to R\pi_*\Omega^d_{\tilde{X}}\]
is a quasi-isomorphism.
\end{cor}

\begin{proof}
We have $\Omega_X^d=R\rho_{X*}\Omega^d_\h=R\pi_* \Omega^d_{\tilde{X}}$ by Propositions \ref{prop_smoothdb} and \ref{prop_topdb}.
\end{proof}

\begin{rem}
Corollary \ref{cor_birationalomega} can be interpreted as a weak version of Grauert-Riemenschneider vanishing, see \cite{GR70} Satz 2.3. It is worth stressing that this vanishing results is {\em not} used in our proofs. A conceptual explanation may be that Kodaira type vanishing results are means of carrying over proper base change from singular cohomology to coherent cohomology. This base change is already incorporated in out $\h$-topology approach.

As the referee pointed out, it is tantalizing to generalize the theory
to positive characteristic. E.g. Corollary \ref{cor_birationalomega} would yield an alternative approach to the main results in \cite{CR11}. However, our approach does use resolution of singularities quite heavily. In characteristic zero, it would actually be possible to use de Jong's weaker version with alterations instead. Nevertheless, a translation to positive characteristic is far from obvious because of the existence of non-separable extensions.
\end{rem}

Recall e.g. from \cite{KM98} Definition 5.8 that a variety $X$ has {\em rational singularities} if
any resolution $\pi:\tilde{X}\to X$ is {\em rational}, i.e., if the adjunction map
\[ \Oh_X\to R\pi_*\Oh_{\tilde{X}} \]
is a quasi-isomorphism. 

\begin{prop}[\cite{KM98} Theorem 5.10]\label{prop_rational}
If some resolution is rational, then so is any other. In particular,
if $\pi:\tilde{X}\to X$ is a proper birational map between smooth varieties, then
\[ \Oh_X\to R\pi_*\Oh_{\tilde{X}}\]
is a quasi-isomorphism.
\end{prop}
\begin{proof}This is a standard consequence of Corollary \ref{cor_birationalomega}. \end{proof}



\section{Application: De Rham cohomology and the Du Bois complex}\label{sec:7}

So far, we have considered the sheaves of $p$-forms separately for every $p$. We now turn our attention to the de Rham complex and show how Du Bois singularities and the relative de Rham cohomology fit into our framework of $\h$-sheaves.

\subsection{Cohomology of the complex of $h$-differentials}

Recall that for smooth varieties, 
$H^i_\Zar(X,\Omega^*_X)$ is called {\em algebraic de Rham cohomology}.
Following Deligne in \cite{HodgeIII}, the definition is extended to arbitrary varieties.

\begin{defn}Let $X\in\Sch$. Let $X_\bullet\to X$ be a proper hypercover such that all $X_n$ are smooth. We
define {\em algebraic de Rham cohomology} of $X$ by
\[ H^i_\dR(X)=H^i_\Zar(X_\bullet,\Omega^*_{X_\bullet})\ .\]
\end{defn}
\begin{rem}
\begin{enumerate}
\item By resolution of singularities such a hypercover exists.
\item If $k=\C$ and $X$ is proper, this is the de Rham component of the mixed
Hodge structure on $H^i_\sing(X^\an,\Q)$ defined by Deligne in \cite{HodgeIII}.
(Here $X^\an$ denotes the analytic space attached to $X$.) In particular,
$H^i_\sing(X^\an,\Q)\tensor_\Q\C$ carries a {\em Hodge filtration}
$F^pH^i_\sing(X^\an,\Q)\tensor_\Q\C$. 

\item For $X$ embeddable into a smooth variety, there is an alternative definition of algebraic de Rham cohomology by Hartshorne,  \cite{Har75}. 
It is in fact equivalent, see Remark \ref{geisserdeRham} below.
\end{enumerate}
\end{rem}

We now turn to $\h$-topology.

\begin{defn} Let $d:\Omega^p_\h\to\Omega^{p+1}_\h$ be the $\h$-sheafification of the exterior differential on $p$-forms. We call
\[ \Omega^*_\h\]
the {\em algebraic de Rham complex in the $\h$-topology}. We denote by
\[ H^i_\h(X,\Omega^*_\h)\]
its hypercohomology in the $\h$-topology. The stupid filtration
\[ F^p\Omega^*_\h=[\dots\to 0\to \Omega^p_\h\to\Omega^{p+1}_\h\to\dots]\subset \Omega^*_\h\]
is called {\em Hodge filtration}. 
\end{defn}

\begin{prop}[(de Rham cohomology)]\label{prop_deRham}
Let $X\in\Sch$. 
Then
\[ H^i_\dR(X)=H^i_\h(X,\Omega^*_\h)\ .\]
\end{prop}
\begin{proof}Let $X_\bullet\to X$ be a proper hypercover with all $X_n$ smooth.
By the hypercohomology spectral sequence, i.e., the spectral
sequence induced by the Hodge filtration on $\Omega^*_\h$, and Proposition \ref{prop_hypercov}
\[ H^i_\h(X,\Omega^*_\h)=H^i_\h(X_\bullet,\Omega^*_\h)=H^i_\Zar(X_\bullet,\Omega^*)\ .\]
By definition, the right hand side is algebraic de Rham cohomology of $X$.
\end{proof}
\begin{rem}\label{geisserdeRham}Geisser in \cite{Gei06} uses the $\eh$-topology instead. (Recall that $\eh$-cohomology and $\h$-cohomology agree.) He shows in loc. cit. Theorem 4.10 that
$\eh$-cohomology of $\Omega^*_\eh$ agrees with Hartshorne's definition of algebraic de Rham cohomology in \cite{Har75}. So indeed, all three approaches give the same result.
\end{rem}

\begin{cor}Let $k=\C$. The embedding induces a natural isomorphism
\[ H^i_\h(X,\Omega^*_\h)\isom H^i_\sing(X_\C^\an,\C)\ .\]
\end{cor}
\begin{proof}This is the period isomorphism between algebraic de Rham
cohomology and singular cohomology.
\end{proof}

\begin{thm}[(Hodge filtration)]\label{thm_degenerates}
Let $X\in\Sch$ be proper. Then the Hodge to de Rham spectral sequence 
\[ E^{ab}_1=H_\h^b(X,\Omega^a_\h)\Rightarrow H^{a+b}_\h(X,\Omega^*_\h)\]
degenerates at $E_1$.
Under the identification in Proposition \ref{prop_deRham}
the Hodge filtration $F^*\Omega^*_\h$ induces the Hodge filtration on $H^i_\dR(X)$. 
\end{thm}
\begin{proof} Choose a proper hypercover $X_\bullet$ of $X$ with all $X_n$ smooth.
The comparison argument in the proof of Proposition \ref{prop_deRham} also works for all $F^p\Omega^*_\h$, hence we can replace $H^i_\h(X,\cdot)$ by
$H^i_\Zar(X_\bullet,\cdot)$ everywhere in the assertion.
Hence this is really a statement about Zariski-cohomology of differential forms. 

By the base change properties of differential forms and cohomology of coherent sheaves, the theorem is true for the ground field $k$ if and only if it is true for a field extension of $k$. 
The scheme $X$ and every $X_n$ is defined over a finitely generated extension of $\Q$. 
Hence we can assume without loss of generality that $k$ is generated over $\Q$ by countably many elements. Such fields can be embedded into $\C$, hence
it suffices to consider the case $k=\C$.

In this case, the assertion is proved by Deligne in \cite{HodgeIII}: by
definition 
\[ F^pH^i_\sing(X^\an,\C)=H^i_\Zar(X_\bullet,F^p\Omega^*)\]
and the spectral sequence attached to the Hodge filtration degenerates at $E_1$.
\end{proof}

\begin{cor}\label{cor_surjection} Let $k=\C$ and $X$ be proper. Then the natural map
\[ H^i_\sing(X^\an,\C)\to H^i_\h(X,\Omega^0_\h)\]
is surjective.
\end{cor}
\begin{proof}In this case we can identify algebraic de Rham cohomology with
singular cohomology.
\end{proof}

\begin{rem} This generalizes the well-known consequence of the degeneration of
the Hodge to de Rham spectral sequence to the singular case. Note that
this is the starting point in the proof of well-known vanishing theorems
in birational geometry, as pointed out e.g. in \cite{Kol87} Part III and \cite{CKM88} Lecture 8.
\end{rem}

\subsection{The Du Bois complex}\label{sec_dubois}
In this subsection we point out how the Du Bois complex fits into the framework of $\h$-differential forms. The connection was first observed by Lee in \cite{Lee09}. 
In order to have a self-contained presentation, we explain the details in our language. We use the new language to reprove a number of results from the literature with very short and simple proofs.

Recall from  \cite{HodgeIII} the notion of a filtered derived category $D^+F\Ah$ of an abelian category $\Ah$. Objects are filtered complexes and morphisms are morphisms of filtered complexes localized with respect to filtered quasi-isomorphisms. We assume that all filtrations are {\em decreasing}.

Du Bois defines in \cite{DuB81} a filtered complex of sheaves on $X$ whose
hypercohomology is de Rham cohomology with the Hodge filtration. We review the construction. 

\begin{defn}Let $X$ be a variety,  $\pi:X_\bullet\to X$ be a proper hypercover with all $X_n$ smooth. Then
\[ \uOmega^*_X:=R\pi_*\Omega^*_{X_\bullet}\]
is the {\em Du Bois complex} of $X$. It is filtered by
\[ F^p\uOmega^*_X=R\pi_*F^p\Omega^*_{X_\bullet}\ .\]
We denote
\[ \uOmega^p_X=F^p\uOmega^*_X/F^{p+1}\uOmega^*_X[p]=R\pi_*\Omega^p_{X_\bullet}\]
its associated graded.
\end{defn}

\begin{rem}Du Bois assumes that $X$ is proper. This is needed if we want the induced filtration on hypercohomology to be the correct Hodge filtration.
\end{rem}

\begin{thm}\label{thm_h-is-dub}
Consider $\Omega^*_\h$ with the Hodge filtration as an object in $D^+F( \Sh(\Sch_\h))$ where $\Sh(\Sch_\h)$ is the category
of sheaves of abelian groups on $\Sch_\h$. Then
\[ R\rho_{X*}\Omega^*_\h\in D^+F(\Sh(X_\Zar))\]
is naturally isomorphic to the Du Bois complex $\uOmega^*_X$. In particular,
\[ R\rho_{X*}\Omega^p_\h\isom \uOmega^p_X\in D(\Sh(X)_\Zar)\ .\]
\end{thm}
\begin{rem}A variant of this result was proved by Lee, see \cite{Lee09} Theorem~4.13.
\end{rem}
\begin{proof}Choose a proper hypercover $\pi:X_\bullet\to X$ as in the
definition of $\uOmega^*_X$. We need to show that
\[ R\rho_{X*}\Omega^*_\h\isom R\pi_*\Omega^*_X\ .\]
Using the Hodge filtration on both sides, this follows from Corollary \ref{prop_rhocohdim}.
\end{proof}

\begin{cor}[(\cite{DuB81})]As object of the filtered derived category, $\uOmega^*_X$ is independent of the choice of hypercover $\pi:X_\bullet\to X$.
\end{cor}
\begin{proof}This is true for $R\rho_{X*}\Omega^*_\h$.\end{proof}
\begin{rem}After this identification, the statements in Section \ref{sec_rho} 
can be understood as statements on the graded pieces of the Du Bois complex.
\end{rem}

\begin{cor}[(\cite{GNPP88} Proposition III.1.17)]
The Du Bois complex of $X$ is a complex of coherent sheaves with $k$-linear coboundary maps concentrated in degrees at most $2\dim X$.
\end{cor}

\begin{proof}
Use Proposition \ref{prop_rhocohdim} and the Hodge spectral sequence.
\end{proof}

 Recall that in \cite{S83} (3.5) Steenbrink defined a variety to have \emph{Du Bois singularities} if the canonical map from $\Oh_X$ to the zeroth graded piece of the Du Bois complex is a quasi-isomorphism. These singularities have already been studied by Du Bois in \cite{DuB81} Section 4.

\begin{cor} A variety $X$ has Du Bois singularities if and only if $\Oh_X\to R\rho_{X*}\Omega^0_\h$ is a quasi-isomorphism, i.e, if
\[ R^i\rho_{X*}\Omega_\h^0=\begin{cases}\Oh_X&i=0\\
                                        0&i>0.
			  \end{cases}\]
In particular, if $X$ has Du Bois singularities, then
\[ H^i_\h(X,\Omega^0_\h)=H^i_\Zar(X,\Oh_X)\ .\]
\end{cor}
\begin{rem} By Proposition \ref{prop_sn}, we have $\rho_{X*}\Omega^0_\h=\Oh_X$ if $X$ is semi-normal. More generally the argument shows that zeroth cohomology of $\uOmega^0_X$ is isomorphic to $\pi_*\Oh_{X^{\text{sn}}}$ where
$\pi:X^\text{sn}\to X$ is the semi-normalization.
This reproves  a result established by Saito  \cite{Sai00} Proposition 5.2 or Schwede \cite{Sch09} Lemma 5.6.
\end{rem}

\begin{ex}[(\cite{DuB81} Exemples 4.7)] 
Normal crossing schemes are Du Bois by Corollary \ref{cor_dnc} for $p=0$.
\end{ex}

Using our approach we can also explain Schwede's criterion. 
\begin{cor}[(\cite{Sch07})]\label{cor_schwede}
Let $X\subset Y$ be a scheme embedded in some smooth scheme $Y$ and let $\pi:\tilde{Y}\to Y$ be a log resolution of $X$ in $Y$, i.e., the exceptional set $E=Exc(\pi)_{red}=\pi^{-1}(X)_{red}$ is a snc divisor. Then $X$ has Du Bois singularities if and only if the canonical map $\Oh_X\to R\pi_*\Oh_E$ is an isomorphism.
\end{cor}

\begin{proof}
By Proposition \ref{prop_smoothdb}, Corollary \ref{cor_dnc} and $R\pi_*\Oh_{\tilde{Y}}=\Oh_Y$ (see Remark~\ref{prop_rational})
the blow-up triangle \ref{prop_blowuptriangle} for $\Oh_\h$ and $(\tilde{Y}\to Y, X\xrightarrow{i} Y)$ can be written as
\[\Oh_Y \to \Oh_Y\oplus i_*R\rho_{X*}\Oh_\h \to R\pi_*\Oh_E \xrightarrow{\text{+1}}  \]
which shows the claim.
\end{proof}


\subsection{Relative de Rham cohomology}\label{ssec_rel-de-rham}
Let $X$ be a variety and $Z\to X$ a closed subvariety. Our aim is to
describe de Rham cohomology of $X$ relative to $Z$ by h-differentials.

As mentioned in the introduction,
relative de Rham cohomology is needed in a generalization of the definition of the period isomorphism and the period numbers of a variety,
e.g. in
the work of Kontsevich and Zagier on periods \cite{KoZ02}, see also \cite{HubMSt11}. 
We are not aware of a reference in the literature even though the existence
of such a theory was clear to the experts.

\begin{defn}Let $X$ be in $\Sch$.
Let $(\Sch/X)_\h$ be the category of schemes of finite type over $X$ equipped
with the restriction of the h-topology. 

Let $f:X\to Y$ be in $\Sch$. We denote
\[ f^*:(\Sch/Y)_\h^\sim\to (\Sch/X)_\h^\sim\]
the restriction functor and by
\[ f_*:(\Sch/X)_\h^\sim\to (\Sch/Y)_\h^\sim\]
its right adjoint.
\end{defn}

We will use frequently the following fact.

\begin{lemma}Let $\pi:X\to\Spec\, k$ be the structural map. For any sheaf $\Fh$ of abelian groups on $\Sch_\h=(\Sch/\Spec\, k)_\h$, we have
\[ H^i_\h(X,\Fh)=H^i((\Sch/X)_\h,\pi^*\Fh)\ .\]
\end{lemma}

\begin{proof}
This is \cite{Milne} Lemma III.1.11.
\end{proof}

\begin{defn}
We denote by $\Omega^p_\Xhtop$ the restriction of $\Omega^p_\h$ to $(\Sch/X)_\h$.
Equivalently, $\Omega^p_\Xhtop$ is the $\h$-sheafification of the presheaf $\Omega^p$ on $\Sch/X$.
\end{defn}

\begin{defn}Let $X\in \Sch$ and $i:Z\to X$ a closed subscheme. Put
\[ \Omega^p_{\basehtop{(X,Z)}}=\Ker(\Omega^p_\Xhtop\to i_*\Omega^p_\Zhtop)\]
in the category of abelian sheaves on $(\Sch/X)_\h$.

We define
{\em relative algebraic de Rham cohomology} as
\[ H^p_\dR(X,Z)=H^p_\h(X,\Omega^*_{\basehtop{(X,Z)}})\ .\]
\end{defn}

\begin{rem}
Relative algebraic de Rham cohomology could alternatively be defined as part 
of the Hodge structure on relative singular cohomology 
or as de Rham realization of appropriate geometric motives (see \cite{Hreal} and \cite{Hreal2}, or \cite{LW09}). The above agrees with these definitions. We refrain from giving the details of the comparison.

Establishing standard properties of relative algebraic de Rham cohomology in
terms of hyperresolutions is very difficult. Indeed, the standard argument would be to compare the situation with singular cohomology, where the proofs are straight-forward. Our approach via the $\h$-topology allows  us to give these
proofs directly.
\end{rem}

\begin{lemma}\label{lemma_relative}Let $i:Z\to X$ be a closed immersion. 
\begin{enumerate}
\item
Then
\[ Ri_*\Omega^p_\Zhtop=i_*\Omega^p_\Zhtop\]
and hence
\[ H^q_\h(X,i_*\Omega^p_\Zhtop)=H^q_\h(Z,\Omega^p_\h)\ .\]
\item The natural map of sheaves of abelian groups on $(\Sch/X)_\h$
\[ \Omega^p_\Xhtop\to i_*\Omega^p_\Zhtop\]
is surjective.
\end{enumerate} 
\end{lemma}
\begin{proof}The higher direct image
$R^qi_*\Omega^p_\Zhtop$ is the $\h$-sheafification in $(\Sch/X)_\h$ of
\[ (f:U\to X)\mapsto H^q_\h(f^{-1}Z,\Omega^p_\h)\] 
We write $g:U'\to U\to X$ where $U'\to U$ is a resolution such that locally on $U'$ the preimage $g^{-1}Z\subset U'$ is either a simple normal crossing divisor or $g^{-1}Z=U'$. In a second step, we cover $U'$ by open affines. Hence it suffices to show that
\[ H^q_\h(Z',\Omega^p_\h)=0 \hspace{3ex}q>0\]
for $Z'$ an affine scheme with normal crossings. By Corollary \ref{cor_dnc}
it is equal to Zariski cohomology of a coherent sheaf, hence zero.

For surjectivity, it suffices by the same reduction to show that
\[ \Omega^p(U')\to\Omega^p_\h(Z')\]
is surjective for $U'$ smooth affine and $Z'\subset U'$ a divisor with normal
crossings. Note that this is true for K\"ahler differentials. By Proposition
\ref{prop_h-dnc}
$\Omega^p_\h|_{Z'}$ is a quotient of $\Omega^p_Z$ in the Zariski topology.
\end{proof}

\begin{prop}[(Long exact sequence)]
Let $Z\subset Y\subset X$ be a closed immersion. Then there is a natural long exact sequence
\[\dots \to H^q_\dR(X,Y)\to H^q_\dR(X,Z)\to H^q_\dR(Y,Z)\to H^{q+1}_\dR(X,Y)\to\cdots\]
\end{prop}
\begin{proof} We compute in the category of abelian sheaves on $(\Sch/X)_\h$. 
Let $i_Y:Y\to X$, $i_Z:Z\to X$.
By definition, there
is a commutative diagram
\[\begin{CD}
 0@>>>\Omega^p_{\basehtop{(X,Y)}}@>>>\Omega^p_\Xhtop @>>>i_{Y*}\Omega^p_\Yhtop@>>> 0\\
@.@VVV@VVV@VVV\\
 0@>>>\Omega^p_{\basehtop{(X,Z)}}@>>>\Omega^p_\Xhtop @>>>i_{Z*}\Omega^p_\Zhtop@>>> 0
 \end{CD}\]
 The functor $i_{Y*}$ is left exact, hence we have
 \[ i_{Y*}\Omega^p_{\basehtop{(Y,Z)}}=\Ker (i_{Y*}\Omega^p_\Yhtop\to i_{Z*}\Omega^p_\Zhtop)\ .\]
 By the snake lemma this implies that we have a short exact sequence
 \[ 0\to \Omega^p_{\basehtop{(X,Y)}}\to \Omega^p_{\basehtop{(X,Z)}}\to i_{Y*}\Omega_{\basehtop{(Y,Z)}}\to 0\ .\]
The long exact sequence for relative cohomology is the long exact cohomology sequence attached to it, provided we establish
\[ Ri_{Y*}\Omega^p_{\basehtop{(Y,Z)}}=i_{Y*}\Omega^p_{\basehtop{(Y,Z)}}\ .\]
Using Lemma \ref{lemma_relative} we obtain a natural triangle
\[ Ri_{Y*}\Omega^p_{\basehtop{(Y,Z)}}\to i_{Y*}\Omega^p_\Yhtop\to i_{Z*}\Omega^p_\Zhtop\ .\]
By considering the long exact cohomology sequence on $(\Sch/X)_\h$, it suffices
to show surjectivity of
\[ i_{Y*}\Omega^p_\Yhtop\to i_{Z*}\Omega^p_\Zhtop\ .\]
This is true because $\Omega^p_\Xhtop$ surjects to both by Lemma \ref{lemma_relative}.
\end{proof}

\begin{prop}[(Excision)]\label{prop_exc}
Let $\tilde{X}\to X$ be a proper morphism, which is an isomorphism outside
of $Z\subset X$. Let $\tilde{Z}=\pi^{-1}Z$. Then 
\[ H^q_\dR(\tilde{X},\tilde{Z})\isom H^q_\dR(X,Z)\ .\]
\end{prop}
\begin{proof}This follows immediately from the blow-up triangle.
\end{proof}
\begin{prop}[K\"unneth formula]\label{prop_kuenneth}
Let $Z\subset X$ and $Z'\subset X'$ be closed immersions. Then there is
a natural isomorphism
\[ H^n_\dR(X\times X',X\times Z'\cup Z\times X')=\bigoplus_{a+b=n}H^a_\dR(X,Z)\tensor_k H^b_\dR(X',Z')\ .\]
\end{prop}

\begin{proof}We work  in the category of $\h$-sheaves of $k$-vector spaces on $X\times X'$. Note that $\h$-cohomology of an $\h$-sheaf of $k$-vector spaces computed 
in the category of sheaves of abelian groups agrees with its $\h$-cohomology 
computed in the category of sheaves of $k$-vector spaces because an injective sheaf of $k$-vector spaces is also injective as sheaf of abelian groups.

We abbreviate $T=X\times Z'\cup Z\times X'$.
By $\h$-sheafification of the product of K\"ahler differentials we have a natural multiplication
\[ \text{pr}_X^*\Omega^a_\basehtop{X}\tensor_k\text{pr}_{X'}^*\Omega^b_\basehtop{X'}\to  \Omega^{a+b}_\basehtop{X\times X'}\ .\]
It induces, with $i_Z:Z\to X$, $i_{Z'}:Z'\to X'$, and $i:T\to X\times X'$
\[ 
   \text{pr}_X^*\Ker(\Omega^a _\basehtop{X}\to i_{Z*}\Omega^a_\basehtop{Z})\tensor_k\text{pr}_{X'}^*\Ker(\Omega^b_\basehtop{X'}\to i_{Z'*}\Omega^b_\basehtop{Z'})\to \Ker( \Omega^{a+b}_\basehtop{X\times X'}\to i_*\Omega^{a+b}_\basehtop{T})\ .\]
 The resulting morphism
 \[ \text{pr}_X^*\Omega^*_{\basehtop{(X,Z)}}\tensor_k \text{pr}_{X'}^*\Omega^*_{\basehtop{(X',Z')}}\to\Omega^*_{\basehtop{(X\times X',T)}}\ . \]
induces a natural K\"unneth morphism
\[ \bigoplus_{a+b=n}H^a_\dR(X,Z)\tensor_k H^b_\dR(X',Z')\to H^n_\dR(X\times X',T)\ . \]
It remains to show that it is an isomorphism. One can easily show that the K\"unneth morphism is compatible with long exact sequences of pairs of spaces in both arguments, possibly up to sign. E.g., for the second argument the diagrams 
\[\begin{CD}
H^a_\dR(X,Z)\otimes H^b_\dR(X',Z')@.\,\to\,@. H^a_\dR(X,Z)\otimes H^b_\dR(X') @.\,\to\,@. H^a_\dR(X,Z)\otimes H^b_\dR(Z') @.\,\to\cdots \\
@VVV@.@VVV@.@VVV\\
 H^{a+b}_\dR(X\times X',T)@.\,\to\,@.H^{a+b}_\dR(X\times X',Z\times X')  @.\,\to\, @. H^{a+b}_\dR(X\times Z',Z\times Z') @. \,\to\cdots
 \end{CD}\]
commute (possibly up to sign), where the second row is the long exact sequence associated with $Z\times X'\subset T\subset X\times X'$  up to the excision isomorphism
\[H^{a+b}_\dR(T,X\times Z')=H^{a+b}_\dR(X\times Z',Z\times Z').\]
By this and a similar consideration for the first argument we reduce to the case $Z'=Z=\emptyset$.

In the second step we use that the K\"unneth morphism is compatible with long exact sequences for abstract blow-ups, again in both arguments. Hence it suffices to show the isomorphism 
\[ \bigoplus_{a+b=n}H^a_\dR(X)\tensor_k H^b(X')\to H^n_\dR(X\times X') \ . \]
for $X$ and $X'$ smooth. In this case we can compute in the Zariski topology. The isomorphism is well-known. It follows from the K\"unneth formula for Zariski cohomology of vector bundles.
\end{proof}

A special case of relative cohomology is cohomology with compact support.
\begin{defn}Let $X\in\Sch$ and $j:X\to\bar{X}$ a compactification. Then
\[ H^q_{\dR,c}(X)=H^q_\dR(\bar{X},X)\]
is called {\em algebraic de Rham cohomology with compact support}.
\end{defn}
In the setting of the eh-topology this is precisely \cite{Gei06} Section 4.1.
By excision, the definition is independent of the choice of compactification.

\end{document}